\documentclass[11pt]{article}
\pdfoutput=1

\usepackage[margin=1.3in]{geometry}

\usepackage{graphicx}
\usepackage{amssymb,amsmath}
\usepackage{amsthm,enumerate}
\usepackage{hyperref,xcolor}
\usepackage{mathrsfs}
\usepackage{extarrows}
\usepackage{textcomp}
\allowdisplaybreaks[4]

\makeatletter

\newdimen\bibspace
\makeatother

\numberwithin{equation}{section}

\newtheorem{theorem}{Theorem}[section]
\newtheorem{lemma}[theorem]{Lemma}

\newtheorem{corollary}[theorem]{Corollary}
\newtheorem{remark}[theorem]{Remark}

\def\<{\langle}
\def\>{\rangle}

\begin{document}

\title{Existence of entire solutions to the Lagrangian mean curvature equations in supercritical phase}

\author{Zixiao Liu, Cong Wang, Jiguang Bao\footnote{Corresponding author J. Bao is supported by the National Key Research and Development Program of China (No. 2020YFA0712904) and the Beijing Natural Science Foundation (No. 1222017). The first author Z. Liu is supported by the China Postdoctoral Science Foundation (No. 2022M720327).}}
\date{\today}

\maketitle

\begin{abstract}
In this paper, we establish the existence and uniqueness theorem of entire solutions to the Lagrangian mean curvature equations with prescribed asymptotic behavior at infinity. The phase functions are assumed to be supercritical and converge to a constant in a certain rate at infinity. The basic idea is to establish uniform estimates for the approximating problems defined on bounded domains and the main ingredient is to construct appropriate subsolutions and supersolutions as barrier functions.
We also prove a nonexistence result to show the convergence rate of the phase functions is optimal.
 \\[1mm]
\\[1mm]
{\textbf{MSC~(2020):}} 35A01;~35J60;~35B40.
\end{abstract}

\section{Introduction}

We mainly focus on entire solutions to the Lagrangian mean curvature equation
\begin{equation}\label{equ-Lagrangian}
  \sum_{i=1}^n\arctan\lambda_i(D^2u)=g,
\end{equation}
where $\lambda_i(D^2u)$ denotes $n$ eigenvalues of the Hessian matrix $D^2u$ and $g$ is a function usually known as \textit{phase} function or \textit{Lagrangian phase} function.
Furthermore, we investigate the  case  $|g|>\frac{(n-2)\pi}{2}$, which provides the concavity of the operator \cite{Yuan-GlobalSolution-SPL} and is referred to  as the supercritical phase condition
\cite{Bhattacharya-Monney-Shankar-GradientEsti-Critical,
Collins-Picard-Wu-ConcavityLagrangian,Lu-Dirichlet-Lagrangian,Yuan-GlobalSolution-SPL}.
Especially when $g$ is a constant, equation \eqref{equ-Lagrangian} is also known as the special Lagrangian equation.
Geometrically, $u$ being a solution to \eqref{equ-Lagrangian} implies that the gradient graph $(x,Du(x))$ has mean curvature $(0,0,\cdots,0,Dg(x))^{\perp}$, see for instance Harvey--Lawson \cite{Harvey-Lawson-CalibratedGeometries} and
Wang--Huang--Bao \cite{Chong-Rongli-Bao-SecondBoundary-SPL} for more details.
We say $u$ is an entire solution if it satisfies equation \eqref{equ-Lagrangian} in the whole Euclidean space $\mathbb R^n.$

Fully nonlinear elliptic equations that rely only on eigenvalues of the Hessian,  such as
\begin{equation}\label{equ-generalEQU}
f(\lambda(D^2u))=g,
\end{equation}
where $\lambda=(\lambda_1,\cdots,\lambda_n)$ is the vector formed by eigenvalues of $D^2u$,
have been an important subject in PDE. Typical examples include  the Monge--Amp\`ere equations, the $k$-Hessian equations, the   Lagrangian mean curvature equations, etc.

For the Dirichlet problems of these equations on bounded domains, exterior domains and entire space, there have been much extensive studies.  The Dirichlet problem on exterior domain infers that finding a solution with prescribed value on inner boundary and near infinity
\begin{equation}\label{equ-prescribe}
u(x)-Q(x)=o(1),\quad\text{as }|x|\rightarrow\infty,\quad\text{where}\quad Q(x):=\dfrac{1}{2}x^TAx+bx+c,
\end{equation}
$x^T$ denotes the transpose of $x$, $A\in\mathtt{Sym}(n)$ is a symmetric $n$ by $n$ matrix, $b\in \mathbb R^n$, and $c\in\mathbb R$ such that $f(\lambda(D^2Q))=f(\lambda(A))=g(\infty)$.
For the Monge--Amp\`ere equations, these problems have been well studied  in a sequence of  work by Caffarelli--Nirenberg--Spruck \cite{Caffarelli-Nirenberg-Spruck-DirichletIII} for bounded domains, see also Savin \cite{Savin-MA-PointWiseBoundaryEst}, Trudinger--Wang \cite{Trudinger-Wang-BoundaryRegularity}, Caffarelli--Tang--Wang \cite{Caffarelli-Tang-Wang-GlobalRegular-MA} and the references therein for further discussions. On exterior domains and entire domains, Dirichlet type problems of the Monge--Amp\`ere equations were considered by Caffarelli--Li \cite{Caffarelli-Li-ExtensionJCP,Caffarelli-Li-Liouville-MA-periodic}, Bao--Li--Zhang \cite{Bao-Li-Zhang-ExteriorBerns-MA,Bao-Li-Zhang-GlobalExterioDP-Dim2}, Li--Lu \cite{Li-Lu-ExistandNonExist}, Bao--Xiong--Zhou \cite{Bao-Xiong-Zhou-ExistenceEntireMA} and the references therein. For the Monge--Amp\`ere equations on half space, Liouville type results, asymptotics at infinity and solvability of Dirichlet type problems, we refer the readers to Savin \cite{Savin-LocalizationThm}, Jia--Li--Li \cite{Jia-Li-Li-AsympMA-halfspace}, Jia--Li \cite{Jia-Li-AsympMA-halfspace}, etc.
For the Lagrangian mean curvature equations, the problems on bounded domains have been partially solved by Collins--Picard--Wu \cite{Collins-Picard-Wu-ConcavityLagrangian}, Battacharya  \cite{Bhattacharya-Dirichlet-LagMeanCurva}, Lu \cite{Lu-Dirichlet-Lagrangian}, Battacharya--Mooney--Shankar \cite{Bhattacharya-Monney-Shankar-GradientEsti-Critical} and the references therein. Exterior Dirichlet problems of the special Lagrangian equations were considered by Li \cite{Li-Dirichlet-SPL}.  Furthermore, we would like to mention that Li--Wang \cite{Li-Wang-EntireDirichlet} obtained solvability of exterior Dirichlet problem of equation \eqref{equ-generalEQU} under some structure conditions on $f$, which includes the Monge--Amp\`ere operators, the $k$-Hessian operators, etc. However, the strategies in the literatures mentioned above are not enough to construct entire solutions to \eqref{equ-Lagrangian}. We overcome the difficulty and develop  new techniques to construct barrier functions, especially on finding  generalized symmetric supersolutions to \eqref{equ-Lagrangian} in entire space with prescribed asymptotic behavior.

Our first result is the existence of solutions to the Lagrangian mean curvature equations in $\mathbb R^n$. Hereinafter, we denote by $\varphi(x)=O_m(|x|^{-\beta}(\ln|x|)^{\zeta})$  the function $\varphi$ satisfying
\[
|D^k\varphi(x)|=O(|x|^{-\beta-k}(\ln|x|)^{\zeta})\quad\text{as }|x|\rightarrow\infty,\quad\forall~k=0,1,\cdots,m,
\]
where $m,\zeta\in\mathbb N$ and $\beta\in\mathbb R$.

\begin{theorem}\label{thm-n>=3}
  Let $n\geq 3$ and $g\in C^{m}(\mathbb R^n)$ satisfy
  \begin{equation}\label{equ-cond-g}
  |g|\in\left(\frac{(n-2)\pi}{2},\frac{n\pi}{2}\right)\quad\text{in }\mathbb R^n,\quad |g(\infty)|>\frac{(n-2)\pi}{2}\quad\text{and}\quad g(x)=g(\infty)+O_m(|x|^{-\beta}),
  \end{equation}
  as $|x|\rightarrow\infty$,
  where  $m\geq 2$ and $\beta>2$. Then for any positive-definite  or negative-definite  matrix $A\in\mathtt{Sym}(n)$ satisfying
\begin{equation}\label{equ-temp-4}
\sum_{i=1}^{n} \arctan \lambda_{i}(A)=g(\infty) \quad \text {and} \quad M(A):=\min_{j,k=1,\cdots,n}\dfrac{1+\lambda_j^2(A)}{2\lambda_k(A)}\cdot \sum_{i=1}^n\dfrac{\lambda_i(A)}{1+\lambda_i^2(A)}>1,
\end{equation}
and any $b\in\mathbb R^n$, $c\in\mathbb R$,
there exists a unique classical solution $u$ to \eqref{equ-Lagrangian} in $\mathbb R^n$
with prescribed asymptotic behavior
\begin{equation}\label{equ-system-Dirich}
    u(x)-\left(\frac{1}{2}x^TAx+bx+c\right)=
    \left\{
  \begin{array}{lllll}
    O_{m+1}(|x|^{2-\min\{\beta,n\}}), & \text{if }\beta\neq n,\\
    O_{m+1}(|x|^{2-n}(\ln|x|)), & \text{if }\beta=n,\\
  \end{array}
  \right.
\end{equation}
as $|x|\rightarrow\infty$. Furthermore,  $u\in C_{loc}^{m+1,\alpha}(\mathbb R^n)$ for any $0<\alpha<1$.
\end{theorem}

By removing a linear function from $u$, hereinafter we may assume without loss of generality that $b=0$ and $c=0$ in \eqref{equ-system-Dirich}. We say a function $u(x)$ is generalized symmetric with respect to $A$ if there exists a scalar function $U$ such that $u(x)=U(\frac{1}{2}x^TAx)$.
Such generalized symmetric functions play an important role in the study of exterior and entire Dirichlet problems, which is pioneered by Bao--Li--Li \cite{Bao-Li-Li-2014}.

The proof of Theorem \ref{thm-n>=3} is separated into three parts. Firstly, we construct generalized symmetric subsolutions and supersolutions in $\mathbb R^n$ with prescribed asymptotic behavior at infinity.
Secondly,
we establish uniform estimates for the approximating problems defined on bounded domains, where the subsolutions and supersolutions obtained in previous step work as barrier functions.
Consequently, we obtain a subsequence of solutions that converges to an entire solution $u_{\infty}$.
Eventually, we finish proving Theorem \ref{thm-n>=3} by removing the linear part of $u_{\infty}$.

Historically for generalized symmetric $u$, estimates of $f(\lambda(D^2u))$ were established by concentrating all terms involving second order derivative $U''$ into one component, such as the following type of estimate in \cite{Li-Wang-EntireDirichlet}
\[
f(\lambda(D^2u))\leq f(a_1U'+(a_n+\delta)x^TAxU'',a_2U',\cdots,a_nU'),
\]
where $a_1\leq a_2\leq \cdots \leq a_n$ are $n$ eigenvalues of $A$ and $\delta>0$. However  as stated in Lemma \ref{lem-supsolu-estimate},  the above estimate holds only for sufficiently large $|x|$ under additional assumptions on $U'$ and $U''$. In order to construct entire supersolutions, we disperse the affection of second order derivatives to all components and obtain the following type estimates as in Lemma \ref{lem-supsolu-estimate-refine},
\[
f(\lambda(D^2u))\leq f(a_1U'+(1+J)a_nx^TAxU'',a_2U'+Ja_nx^TAxU'',\cdots,a_nU'+Ja_nx^TAxU''),
\]
where $J$ is an explicit function relying on $x, A, U', U''$ and $f$. The construction of $J$ is technical to make sure the supersolution  has prescribed asymptotic behavior at infinity.

\begin{remark}
 When $A=\tan\frac{g(\infty)}{n} I$, where $I$ denotes the identity matrix, condition \eqref{equ-temp-4} holds and we have a solution that is asymptotically radially symmetric.
 When $A$ is positive-definite or negative-definite and the minimum or maximum eigenvalue has multiplicity no less than $2$, the condition $M(A)>1$ in \eqref{equ-temp-4} holds.
 Furthermore, when $n=3$ and $g(\infty)\in (\frac{\pi}{2},\pi)$, we consider diagonal matrix $A$ such that
  \[
  \lambda_1(A)=\epsilon>0,\quad\lambda_2(A)=\lambda_3(A)=\tan\left(\frac{g(\infty)-\arctan\epsilon}{2}\right).
  \]
 In this case, the equality in \eqref{equ-temp-4} holds but
  \[
M(A)=\dfrac{1+\epsilon^2}{2\tan\left(\frac{
g(\infty)-\arctan\epsilon}{2}
\right)}\cdot \left(\dfrac{2\tan\left(\frac{
g(\infty)-\arctan\epsilon}{2}
\right)}{1+\left(\tan\left(\frac{
g(\infty)-\arctan\epsilon}{2}
\right)\right)^2}+\dfrac{\epsilon}{1+\epsilon^2}\right)<1
\]
  for sufficiently small $\epsilon$. Hence there are matrixes $A$ fail to satisfy $M(A)>1$ in \eqref{equ-temp-4}. Whether existence results hold in such cases remains a difficult problem, see also similar assumptions on exterior Dirichlet problems in \cite{Li-Wang-EntireDirichlet,Li-Dirichlet-SPL}, etc.
\end{remark}

Our second result focuses on the case when $0<\beta\leq 2$ in \eqref{equ-cond-g}.
We show that the existence result may fail by considering radially symmetric solutions.

\begin{theorem}\label{thm-Optimality}
  Let $n\geq 3$.
  For any $0<\beta\leq 2$ and $m\geq 2$, there exist $g\in C^m(\mathbb R^n)$ satisfying condition \eqref{equ-cond-g}   and $A\in\mathtt{Sym}(n)$ satisfying \eqref{equ-temp-4}
  such that there is no classical solution to \eqref{equ-Lagrangian} in $\mathbb R^n$ satisfying
  \begin{equation}\label{equ-ContraAssum}
  u(x)-\dfrac{1}{2}x^TAx=o(1),\quad\text{as }|x|\rightarrow\infty.
  \end{equation}
\end{theorem}

Eventually, we would like to mention that the condition \eqref{equ-system-Dirich} origins from the asymptotic behavior results of solutions on entire $\mathbb R^n$ or exterior domain. There are generous results on this topic, see for instance
Caffarelli--Li \cite{Caffarelli-Li-ExtensionJCP} and Bao--Li--Zhang \cite{Bao-Li-Zhang-ExteriorBerns-MA} for the Monge--Amp\`ere equations,
Li--Li--Yuan \cite{Li-Li-Yuan-BernsteinThm} for the special Lagrangian equations, Liu--Bao \cite{Liu-Bao-2021-Expansion-LagMeanC,Liu-Bao-2021-Dim2,Liu-Bao-2020-ExpansionSPL,Liu-Bao-2021-Dim2-MeanCur}
for a family of mean curvature equations of gradient graphs and Jia \cite{Jia-Xiaobiao-AsymGeneralFully} for a family of general fully nonlinear elliptic equations under asymptotic assumptions of the Hessian matrix.

The paper is organized as follows. In sections \ref{sec-ConstrucSubSol} and \ref{sec-ConstrucSupSol}, we construct generalized symmetric subsolutions and supersolutions.
In sections \ref{sec-ProofThm1} and \ref{sec-ProofThm2}, we prove Theorems \ref{thm-n>=3} and \ref{thm-Optimality} respectively.

\section{Construction of subsolution}\label{sec-ConstrucSubSol}

In this section, we show the existence of subsolution to \eqref{equ-Lagrangian} in $\mathbb R^n$ with prescribed asymptotic behavior at infinity. Hereinafter, we let $a:=(a_1,\cdots,a_n):=\lambda(A)$   and
\[
F(D^2u):=f(\lambda(D^2u)):=\sum_{i=1}^n\arctan \lambda_i(D^2u).
\]
We may assume without loss of generality that $g(x)>\frac{(n-2)\pi}{2}$ and $A$ is positive-definite satisfying \eqref{equ-temp-4}, otherwise we consider  $-u$ instead. Let
\[
s(x):=\frac{1}{2}x^TAx=\frac{1}{2}\sum_{i=1}^na_ix_i^2,\quad\forall~x\in\mathbb R^n.
\]
Then $u(x)$ is generalized symmetric with respect to $A$ if there exists a scalar function $U$ such that
\begin{equation}\label{equ-def-generalizedSymm}
u(x)=U(s(x)),\quad\forall~x\in\mathbb R^n.
\end{equation}
We refer to \cite{User'sGuide-ViscositySol} for the definition of  viscosity subsolutions and supersolutions to \eqref{equ-Lagrangian}. The main result in this section can be stated as the following.
\begin{lemma}\label{lem-subsolution-Existence}
  For any $\beta>2$ and positive-definite $A\in\mathtt{Sym}(n)$ satisfying condition \eqref{equ-temp-4},
   there exists a function  $\underline u$  generalized symmetric to $A$ and is a viscosity subsolution to \eqref{equ-Lagrangian} in $\mathbb R^n$ with asymptotic behavior
  \begin{equation}\label{equ-temp-2}
  \underline u(x)=\dfrac{1}{2}x^TAx+\left\{
  \begin{array}{lllll}
    O(|x|^{2-\min\{2M(A),\beta\}}), & \text{if }M(A)\neq\frac{\beta}{2},\\
    O(|x|^{2-\beta}\ln |x|), & \text{if }M(A)=\frac{\beta}{2}.
  \end{array}
  \right.
  \end{equation}
\end{lemma}

We may assume without loss of generality that
\[
0<a_1\leq a_2\leq\cdots\leq a_n,
\]
otherwise we only need to apply an additional orthogonal transformation to $u$.
By a direct computation, the gradient and Hessian matrix of generalized symmetric function $u$ are composed of
\[
D_iu(x)=a_ix_iU'(s)\quad\text{and}\quad D_{ij}u(x)=a_i\delta_{ij}U'(s)+a_ia_jx_ix_jU''(s),
\]
where $\delta_{ij}$ denotes the Kronecker delta symbol. Especially,
\begin{equation}\label{equ-trace}
\Delta u=\sum_{i=1}^n\lambda_i(D^2u)=\left(\sum_{i=1}^na_i\right)U'+\sum_{j=1}^na_j^2x_j^2U''.
\end{equation}
By the Weyl's theorem, we have the following estimates on the eigenvalues of Hessian matrix of generalized symmetric functions, see for instance  \cite{Book-Horn-Johnson-MatrixAnalysis} and Lemmas 2.1 and 2.3 in \cite{Li-Wang-EntireDirichlet}.
\begin{lemma}\label{lem-subsolu-estimate}
  Let $u$ be a $C^2$ function generalized symmetric with respect to $A$ and $U$ be the scalar function such that \eqref{equ-def-generalizedSymm} holds.  Assume $U'>0$ and $U''\leq 0$. Then
\begin{equation}\label{equ-EstimateEigenvalue}
a_{i} U^{\prime}(s)+\sum_{j=1}^{n} a_{j}^{2} x_{j}^{2} U^{\prime \prime}(s) \leq \lambda_{i}\left(D^{2} u(x)\right) \leq a_{i} U^{\prime}(s), \quad \forall~ 1 \leq i \leq n,
\end{equation}
and
\begin{equation}\label{equ-LowerEstimate-F}
F(D^2u)\geq f(\overline a),\quad\text{where}\quad\overline a:=\left(a_1U'+\sum_{j=1}^na_j^2x_j^2U'',a_2U',\cdots,a_nU'\right),
\end{equation}
as long as all components of $\overline a$ are positive.
\end{lemma}

Before proving Lemma \ref{lem-subsolu-estimate}, we introduce the following calculus fact.

\begin{lemma}\label{lem-temp-1}
  Let $b_1,b_2\geq 0$  and $\delta\leq 0$ satisfy
  \[
  0\leq b_1\leq b_2\quad\text{and}\quad  b_1+\delta\geq 0.
  \]
  Then
  \[
  \arctan(b_1+\delta)+\arctan b_2\leq \arctan b_1+\arctan (b_2+\delta).
  \]
\end{lemma}
\begin{proof}

  Let
  \[
  \xi(t):=\arctan(t+\delta)-\arctan t,\quad t\geq -\frac{\delta}{2}.
  \]
  By a direct computation,
  \[
  \xi'(t)=\dfrac{1}{1+(t+\delta)^2}-\dfrac{1}{1+t^2}=
  \dfrac{-\delta (2t+\delta)}{(1+(t+\delta)^2)(1+t^2)}\geq 0.
  \]
  Consequently, $\xi$ is monotone nondecreasing and we have
  \[
  \xi(b_1)\leq \xi(b_2),\quad\text{i.e.,}\quad
  \arctan (b_1+\delta)-\arctan b_1\leq \arctan (b_2+\delta)-\arctan b_2.
  \]
  This   finishes the proof immediately.
\end{proof}

\begin{proof}[Proof of Lemma \ref{lem-subsolu-estimate}]
We provide a short proof here since the assumptions are slightly different from the one in \cite{Li-Wang-EntireDirichlet}. Estimate \eqref{equ-EstimateEigenvalue} can be bound as Lemma 2.1 in
\cite{Li-Wang-EntireDirichlet}.

By \eqref{equ-trace} and \eqref{equ-EstimateEigenvalue},
 there exist $0\leq \theta_i(s)\leq 1, i=1,\cdots,n$ such that
\begin{equation}\label{equ-temp-15}
\lambda_i(D^2u)=a_iU'+\theta_i\sum_{j=1}^na_j^2x_j^2U''\quad\text{and}\quad\sum_{i=1}^n\theta_i=1.
\end{equation}
Then the desired result follows by repeatedly using  Lemma \ref{lem-temp-1}. In fact, by applying Lemma \ref{lem-temp-1} with respect to the first two variables of $f$, we have
\[
\begin{array}{llll}
  F(D^2u)&=&\displaystyle
  f\left(a_1U'+\theta_1V,
  a_2U'+\theta_2V
  ,a_3U'+\theta_3V,\cdots,a_nU'+\theta_n V\right)
  \\
  &\geq&\displaystyle  f\left(a_1U'+(\theta_1+\theta_2)V,a_2U',a_3U'+\theta_3V,\cdots,a_nU'+\theta_nV\right),
\end{array}
\]
where $V:=\sum_{j=1}^na_j^2x_j^2U''\leq 0$ and we used the fact that
\[
0<a_1U'+\theta_1V\leq a_1U'\leq a_2U'.
\]
Repeat the argument finite times and  we finish the proof.
\end{proof}

By condition \eqref{equ-cond-g}, there exist monotone smooth functions $\underline g$ and $\overline g$ on $[0,+\infty)$ such that
\begin{equation}\label{equ-temp-3}
\frac{(n-2)\pi}{2}< \underline g(s(x))\leq g(x)\leq \overline g(s(x)),\quad\forall~x\in \mathbb R^n
\end{equation}
and
\[
\underline g(s),~\overline g(s)=g(\infty)+O_m(s^{-\frac{\beta}{2}})\quad\text{as }s\rightarrow\infty.
\]
By the monotonicity of $\arctan$ function, there exists a unique decreasing positive function $\underline w(s)$ defined on $[0,+\infty)$ determined by
\[
f\left(a_{1} \underline w(s), \cdots, a_{n} \underline w(s)\right)=\overline{g}(s).
\]
\begin{lemma}\label{lem-ImplicitFuncResult}
There exists a unique smooth function  $h(s,w)$ satisfying
\[
f(h(s,w),a_2w,\cdots,a_nw)=\overline g(s)\quad\text{in }\{(s,w)~|~s\geq 0,~w\geq\underline w(s)\}.
\]
Furthermore, $h(s,\underline w(s))=a_1\underline w(s)$, $h(s,w)$
is monotone decreasing with respect to $s$ and $w$, and there exists $C>0$ such that
\begin{equation}\label{equ-temp-19}
|h(s,w)-h(\infty,w)|\leq Cs^{-\frac{\beta}{2}}\quad\text{and}\quad \dfrac{1}{2a_n}\left(\dfrac{\partial h}{\partial w}(\infty,1)-a_1\right)=-M(A),
\end{equation}
  where the constant $M(A)$ defined in \eqref{equ-temp-4} may be written as
  \[
  M(A)=\dfrac{1+a_1^2}{2a_n}\sum_{i=1}^n\dfrac{a_i}{1+a_i^2}.
  \]
\end{lemma}

\begin{proof}
  Notice that for all $w>\underline w(s)$, we have
  \[
  \lim_{h\rightarrow a_1w}f(h,a_2w,\cdots,a_nw)=f(a_1w,a_2w,\cdots,a_nw)>\overline g(s)
  \]
  and
  \[
  \lim_{h\rightarrow -\infty}f(h,a_2w,\cdots,a_nw)<\frac{(n-2)\pi}{2}<\overline g(s).
  \]
  Hence by the mean value theorem, there exists a unique function $h(s,w)$ such that
  \[
  f(h(s,w),a_2w,\cdots,a_nw)=\overline g(s).
  \]
  Especially, for all $s\geq 0$ and $w\geq \underline w(s)$,
  \[
  \tan\left( g(\infty)-\frac{(n-1)\pi}{2}\right)\leq h(s,w)\leq a_1\underline w(s),
  \]
  which implies  $h(s,w)$  is a bounded function.

  By the implicit function theorem, $h$ is smooth,
  \[
  \frac{\partial f}{\partial \lambda_1}(h(s,w),a_2w,\cdots,a_nw)\cdot \frac{\partial h}{\partial s}(s,w)=\frac{\partial\overline g}{\partial s}(s)<0
  \]
  and
  \[
  \frac{\partial f}{\partial \lambda_1}(h,a_2w,\cdots,a_nw) \cdot \frac{\partial h}{\partial w}
  +\sum_{i=2}^na_i\frac{\partial f}{\partial\lambda_i}(h,a_2w,\cdots,a_nw)=0.
  \]
  Consequently $h(s,w)$ is monotone decreasing with respect to both $s$ and $w$.
  Sending $(s,w)$ to $(\infty,1)$, it follows immediately that
  \[
  \dfrac{1}{1+a_1^2}\cdot\dfrac{\partial h}{\partial w}(\infty,1)+\sum_{i=2}^n\dfrac{a_i}{1+a_i^2}=0,\quad\text{i.e.,}\quad
  \dfrac{\partial h}{\partial w}(\infty,1)=-(1+a_1^2)\sum_{i=2}^n\dfrac{a_i}{1+a_i^2}.
  \]
  This proves the second equality in \eqref{equ-temp-19} immediately and it remains to prove the first equality.

  Since $h(s,w)$ is a bounded smooth function, there exists $C>0$ such that
  \[
  \begin{array}{llll}
    \overline g(s)-g(\infty)&=& f(h(s,w),a_2w,\cdots,a_nw)-f(h(\infty,w),a_2w,\cdots,a_nw)\\
    &=&\arctan h(s,w)-\arctan h(\infty,w)\\
    &\geq & C(h(s,w)-h(\infty,w)).
  \end{array}
  \]
  Hence the first equality in \eqref{equ-temp-19} follows from the asymptotic behavior of $\overline g$.
\end{proof}

\begin{corollary}\label{Coro-temp1}
  Let $u, U$ and $A$ be as in Lemma \ref{lem-subsolu-estimate}.  If
  \[
U'>\underline w,\quad U''\leq 0\quad\text{and}\quad
a_1U'+2a_nsU''\geq h_0(s,U')\quad\text{in }s>0,
\]
where $h_0(s,w):=\max\{0,h(s,w)\}$,
then $u$ is a subsolution to \eqref{equ-Lagrangian} in $\mathbb R^n\setminus\{0\}$.
\end{corollary}
\begin{proof}
By Lemmas \ref{lem-subsolu-estimate} and \ref{lem-ImplicitFuncResult},
  \[
  F(D^2u)\geq f\left(a_1U'+\sum_{j=1}^na_j^2x_j^2U'',a_2U',\cdots,a_nU'\right)\geq f(a_1U'+2a_nsU'',a_2U',\cdots,a_nU')
  \]
  in $\mathbb R^n\setminus\{0\}$. Consequently by the inequalities satisfied by $U$, we continue to have
\[
F(D^2u)\geq f(h_0(s,U'),a_2U',\cdots,a_nU')\geq f(h(s,U'),a_2U',\cdots,a_nU')=\overline g(s)
\]
and this finishes the proof.
\end{proof}

To avoid the singularity at $s=0$, we construct solution to
\[
a_1U'+2a_n(s+1)U''=h_0(s,U')
\]
instead of
\[
a_1U'+2a_nsU''=h_0(s,U').
\]

\begin{lemma}\label{lem-ODEsolution-subsol}
  For any $w_0>\underline w(0)$, there exists a unique solution $W$ to
  \begin{equation}\label{equ-system-Dirich-subsol}
  \left\{
  \begin{array}{llllll}
    \displaystyle \dfrac{\mathrm{d} w}{\mathrm{d} s}=\dfrac{h_0(s,w)-a_1w}{2a_n(s+1)}, & \text{in } s>0,\\
    w(0)=w_0.\\
  \end{array}
  \right.
  \end{equation}
  Furthermore,
  \[
  W'(s)\leq 0,\quad \underline w(s)<W(s)\quad\text{in }s>0,
  \]
  and
  \begin{equation}\label{equ-temp-14}
  W(s)=1+\left\{
  \begin{array}{llllll}
    O(s^{-\min\{M(A),\frac{\beta}{2}\}}), & \text{if }M(A)\neq\frac{\beta}{2},\\
    O(s^{-\frac{\beta}{2}}\ln s), & \text{if }M(A)=\frac{\beta}{2},\\
  \end{array}
  \right.
  \end{equation}
  as $s\rightarrow\infty$.
\end{lemma}

\begin{proof}
  Since $w_0>\underline w(0)$, by the smoothness of $\underline w$, there exists $s_0>0$ such that
  \[
  [0,s_0]\times\left[\frac{w_0+\underline w(0)}{2},\frac{3w_0-\underline w(0)}{2}\right]
  \subset \{(s,w)~|~s\geq 0,~w\geq\underline w(s)\}.
  \]
  Since $h(s,w)$ is a smooth function,
  the right hand side term $\frac{h_0(s,w)-a_1w}{2a_n(s+1)}$ is Lipschtiz in the rectangle  $[0,s_0]\times\left[\frac{w_0+\underline w(0)}{2},\frac{3w_0-\underline w(0)}{2}\right]$.
  By the existence and uniqueness theorem of ODE, such as the Picard--Lindel\"of theorem,  the
  initial value problem \eqref{equ-system-Dirich-subsol} admits locally a unique solution $W$ near $s=0$ and we shall prove that the solution can be extended to $s\in[0,+\infty)$.

  By Lemma \ref{lem-ImplicitFuncResult}, for all $s\geq 0$ and $w> \underline w(s)$,
  \[
  h(s,w)\leq a_1\underline w<a_1w\quad\text{and hence}\quad
  h_0(s,w)-a_1w<0.
  \]
  By the equation in \eqref{equ-system-Dirich-subsol},  $W$ is monotone decreasing as long as $W>\underline w$. Next,
  since $W(0)>\underline w(0)$,
  we claim that $W$ cannot touch $\underline w(s)$ from above.
  Arguing by contradiction, we suppose  there exists $s_0>0$ such that
  \[
  W(s_0)=\underline w(s_0)\quad\text{and}\quad W(s)>\underline w(s)\quad\text{in }[0,s_0).
  \]
  Then from the definition of derivative and the equation in \eqref{equ-system-Dirich-subsol}, we have
  \[
  W'(s_0)=\lim_{s\rightarrow s_0^-}\dfrac{W(s)-W(s_0)}{s-s_0}\leq \underline w'(s_0)\quad\text{but}\quad
  W'(s_0)=\dfrac{h_0(s_0,W(s_0))-a_1W(s_0)}{2a_n(s_0+1)}\geq 0.
  \]
  This becomes  a contradiction since $\underline w$ is a monotone decreasing function.
  Combining the results above, $W$ is monotone decreasing and  $W(s)>\underline w(s)$ holds as long as $W(s)$ exists. By the Carath\'eodory extension theorem of ODE, the solution $W$ exists on entire $[0,+\infty)$.

  Now we prove the asymptotic behavior of $W$ near infinity.
  Since $W$ is monotone decreasing and bounded from below by $\underline w(s)\geq 1$, $W$ admits a finite limit $W(\infty)\geq 1$ at infinity and we claim that $W(\infty)=1$. Arguing by contradiction, if $W(\infty)>1$, then there exists $\epsilon>0$ such that
  \[
  \dfrac{\mathrm{d} W(s)}{\mathrm{d} s}=\dfrac{h_0(s,W(s))-a_1W(s)}{2a_n(s+1)}\geq \frac{\epsilon }{s}\quad\text{in }s>1.
  \]
  This contradicts to the fact that $W$ converge to $W(\infty)$ at infinity. Next, we refine the asymptotic behavior by setting
  \[
  t:=\ln (s+1)\in (0,+\infty)\quad\text{and}\quad \varphi(t):=W(s(t))-1.
  \]
  By a direct computation, for all $t\in (0,+\infty)$,
  \[
  \begin{array}{lllll}
  \varphi'(t)&=& W'(s(t))\cdot e^t\\
  &=&\displaystyle \dfrac{h_0(s(t),\varphi+1)-a_1(\varphi+1)}{2a_n}\\
  &=&\displaystyle
  \dfrac{h_0(s(t),\varphi+1)-h_0(\infty,\varphi+1)}{2a_n}
  +\dfrac{h_0(\infty,\varphi+1)-a_1(\varphi+1)}{2a_n}\\
  &=:& h_1(t,\varphi)+ h_2(\varphi).\\
  \end{array}
  \]
  Furthermore, by \eqref{equ-temp-19}, there exists $C>0$ such that for all $t\gg 1$ and $|\varphi|\ll 1$,
  \[
  |h_1(t,\varphi)|\leq Ce^{-\frac{\beta}{2}t},\quad
  \dfrac{\mathrm{d} h_2}{\mathrm{d} \varphi}(0)=-M(A)\quad\text{and}\quad
  \left| h_2(\varphi)-\dfrac{\mathrm{d} h_2}{\mathrm{d} \varphi}(0)\cdot\varphi\right|
  \leq C\varphi^2.
  \]
   Consequently, $\varphi$ satisfies
  \begin{equation}\label{equ-temp-12}
  \varphi'(t)=-M(A)\varphi+O(e^{-\frac{\beta}{2}t})+O(\varphi^2)
  \end{equation}
  as $t\rightarrow\infty$ and $\varphi\rightarrow 0.$
  By the asymptotic stability of ODE (see for instance Theorem 1.1 of Chap.13 in \cite{Book-Coddington-Levinson-ODE} or Theorem 2.16 in \cite{Book-Bodine-Lutz-AsymptoticIntegration}), we have
  \begin{equation}\label{equ-temp-13}
  \varphi(t)=\left\{
  \begin{array}{llll}
    O(e^{-M(A)t}), & \text{if }M(A)<\frac{\beta}{2},\\
    O(te^{-\frac{\beta}{2}t}), & \text{if }M(A)=\frac{\beta}{2},\\
    O(e^{-\frac{\beta}{2}t}), & \text{if }M(A)>\frac{\beta}{2}.
  \end{array}
  \right.
  \end{equation}
  More explicitly, since $\varphi>0$, for any  sufficiently small $\epsilon>0$, $\varphi$ satisfies
  \[
  \varphi'\leq -(M(A)-\epsilon)\varphi+Ce^{-\frac{\beta}{2}t},\quad\forall~t>T_0,
  \]
  for some $C, T_0$ sufficiently large. Multiplying both sides by $e^{(M(A)-\epsilon)t}$ and taking integral over $(T_0,t)$, there exists $C>0$ such that
  \[
  \varphi\leq \left\{
  \begin{array}{lll}
    Ce^{-\min \{M(A)-\epsilon,\frac{\beta}{2}\}t}, & \text{if }M(A)-\epsilon\neq \frac{\beta}{2},\\
    Cte^{-\min \{M(A)-\epsilon,\frac{\beta}{2}\}t}, & \text{if }M(A)-\epsilon= \frac{\beta}{2},\\
  \end{array}
  \right.
  \]
  for $t>T_0$. Putting this estimate into equation \eqref{equ-temp-12} and choosing sufficiently small $\epsilon$, we have
  \[
  \varphi'\leq -M(A)\varphi+Ce^{-\frac{\beta}{2}t}+
  \left\{
  \begin{array}{lll}
    Ce^{-2\min \{M(A)-\epsilon,\frac{\beta}{2}\}t}, & \text{if }M(A)-\epsilon\neq \frac{\beta}{2},\\
    Ct^2e^{-2\min \{M(A)-\epsilon,\frac{\beta}{2}\}t}, & \text{if }M(A)-\epsilon= \frac{\beta}{2},\\
  \end{array}
  \right.\quad\forall~t>T_0',
  \]
  for some $C, T_0'$ sufficiently large. Multiplying both sides by $e^{M(A)t}$ and taking integral over $(T_0,t)$, we have the desired estimate \eqref{equ-temp-13}.
  This proves the desired asymptotic behavior of $W$.
\end{proof}

\begin{proof}[Proof of Lemma \ref{lem-subsolution-Existence}]
  Take any $w_0>\underline w(0)$ and $W$ to be the solution to \eqref{equ-system-Dirich-subsol} from Lemma \ref{lem-ODEsolution-subsol}. Let
  \[
  \underline u(x):=\underline U(s(x)):=\int_0^{s}W(r)\mathrm dr+C,
  \]
  where $C$ is a constant to be determined.

  Firstly, we prove $\underline u$ is a subsolution in $\mathbb R^n\setminus\{0\}$. By Corollary \ref{Coro-temp1}, for all $x\in\mathbb R^n\setminus\{0\}$,
  \[
  \begin{array}{lllll}
  F(D^2\underline u)&\geq&
  f(a_1\underline U'+2a_n(s+1)\underline U'',a_2\underline U',\cdots,a_n\underline U')\\
  &=& f(h_0(s,W),a_2W,\cdots,a_nW)\\
  &\geq &\overline g(s).\\
  \end{array}
  \]
  Consequently $\underline u$ is a subsolution to \eqref{equ-Lagrangian} in $\mathbb R^n\setminus\{0\}$ as long as all components of $\overline a$ are positive. In fact, for all $x\in\mathbb R^n\setminus\{0\}$,
  \[
  \begin{array}{llllll}
    0&\leq & h_0(s,W)\\
    &=& a_1W+2a_n(s+1)W'\\
    &<&\displaystyle a_1W+\sum_{j=1}^na_j^2x_j^2W',
  \end{array}
  \]
  and this finishes the proof of this part.

  Secondly, we compute the asymptotic behavior at infinity and choose $C$. By Lemma \ref{lem-ODEsolution-subsol}, since $M(A)>1$ and $\beta>2$, we may choose
  \[
  C:=
  \lim_{s\rightarrow\infty}\int_0^{s}(1-W(r))\mathrm dr.
  \]
  This makes $\underline u$ satisfies the desired asymptotic behavior at infinity.

  Eventually, noticing that
  \[
  D\underline u(x)=\underline U'(s(x))\cdot Ax\rightarrow 0\quad\text{as }x\rightarrow 0,
  \]
  by Theorem 1.1 in \cite{Caffarelli-Li-Nirenberg--ViscositySol-III}, the functions $\underline u$ is a subsolution to \eqref{equ-Lagrangian} in $\mathbb R^n$ in viscosity sense.
\end{proof}

\section{Construction of supersolution}\label{sec-ConstrucSupSol}

In this section, we prove the existence of supersolution with prescribed asymptotic behavior at infinity.

\begin{lemma}\label{lem-supersolution-Existence}
  For any $\beta>2$ and positive-definite $A\in\mathtt{Sym}(n)$ satisfying \eqref{equ-temp-4}, there exists a function $\overline u$ generalized symmetric to $A$ and is a viscosity supersolution to \eqref{equ-Lagrangian} in $\mathbb R^n$ with asymptotic behavior as in \eqref{equ-temp-2}.
\end{lemma}

\begin{remark}
  When $g(x)\geq g(\infty)$,  $\frac{1}{2}x^TAx$ is a supersolution to \eqref{equ-Lagrangian} in $\mathbb R^n$ with desired asymptotic behavior. In this case, we may take $\overline u(x):=\frac{1}{2}x^TAx$ as the supersolution directly.
\end{remark}

Different from the construction of subsolution, when $U'>0$ and $U''\geq 0$,  inequality
\[
F(D^2u)\leq f(\bar a),\quad\text{where}\quad\overline a:=\left(a_1U'+\sum_{j=1}^na_j^2x_j^2U'',a_2U',\cdots,a_nU'\right),
\]
fails in general, especially when all $a_i$ are the same. We would like to mention that in Lemmas 3.1 and 3.2 in \cite{Li-Wang-EntireDirichlet}, the authors proved the following estimate of $F(D^2u)$ at infinity under additional assumptions on $U'$ and $U''$.
\begin{lemma}\label{lem-supsolu-estimate}
  Let $\delta>0$,  $u, U$ and $A$ be as in Lemma \ref{lem-subsolu-estimate}.
  Assume $U'>0$ and $U''\geq 0$, then
\begin{equation}\label{equ-EstimateEigenvalue-2}
a_{i} U^{\prime}(s) \leq \lambda_{i}\left(D^{2} u(x)\right) \leq a_{i} U^{\prime}(s)+\sum_{j=1}^{n} a_{j}^{2} x_{j}^{2} U^{\prime \prime}(s), \quad \forall~ 1 \leq i \leq n,
\end{equation}
If in addition,
\[
\lim_{s\rightarrow+\infty}U'=1\quad\text{and}\quad\lim_{s\rightarrow+\infty}sU''(s)=0.
\]
Then there exists $\bar s=\bar s(A,\delta, U',U'')>0$ such that for any $s>\bar s$,
\[
F(D^2u)\leq f(\bar a_{\delta}),\quad\text{where}\quad\bar a_{\delta}:=\left(
a_1U'+(2a_n+\delta)sU'', a_2U',\cdots,a_nU'
\right).
\]
\end{lemma}

In this section, we provide the following estimate of $F(D^2u)$ that works for all $s>0$ under weaker assumptions on $U'$ and $U''$.
\begin{lemma}\label{lem-supsolu-estimate-refine}
  Let   $u$, $U$, $A$ be as in Lemma \ref{lem-subsolu-estimate}. Assume  $U'>0$ and $U''\geq 0$.
  Then there exists a nonnegative smooth function $J=J(U',(s+1)U'')$ such that
  \[
  F(D^2u)\leq f(\bar a_J),
  \]
  where
  \[
  \bar a_J:=\left(
  a_1U'+(1+J)\cdot 2a_nsU'',a_2U'+J\cdot 2a_nsU'',\cdots,a_nU'+J\cdot 2a_nsU''
  \right).
  \]
  More explicitly,  we may choose
  \begin{equation}\label{equ-temp-16}
  J(w,H):=\dfrac{(\sqrt{1+(a_1w)^2}+4a_1H)^2\cdot (a_1w+(a_1w)^3+4a_nH)}{(1+(a_1w)^2)^2\cdot a_1w}-1,
  \end{equation}
  which is monotone increasing with respect to $H$  and satisfies $J(w,0)=0$.
\end{lemma}
\begin{proof} By \eqref{equ-trace} and \eqref{equ-EstimateEigenvalue-2}, there exists $0\leq \theta_i(s)\leq 1$, $i=1,\cdots,n$ such that \eqref{equ-temp-15} holds.
  Let
  \[
  V:=\sum_{j=1}^na_j^2x_j^2U''\quad\text{and}\quad
  \widehat a_J:=\left(a_1U'+(1+J)V,a_2U'+JV,\cdots,a_nU'+JV\right).
  \]
  Especially, for all $s>0$,
  \begin{equation}\label{equ-temp-17}
  0\leq 2a_1sU''\leq V(x)\leq 2a_nsU''.
  \end{equation}
Obviously when $J\geq 1$, we always have
\[
F(D^2u)<f(\widehat a_J)\leq f(\bar a_J).
\]
Hence we only need prove the desired result under the case of $J(U',(s+1)U'')<1.$

  By the Newton--Leibnitz formula, for any $s>0$,
  \[
  \begin{array}{lll}
    F(D^2u)-f(\widehat a_J)&=& f(a_1U'+\theta_1V,a_2U'+\theta_2V,\cdots,a_nU'+\theta_nV)-f(\widehat a_J)\\
    &=&\displaystyle -\int_{\theta_1V}^{(1+J)V}\dfrac{1}{1+(a_1U'+t)^2}\mathrm dt
    +\sum_{j=2}^n\int^{\theta_jV}_{JV}\dfrac{1}{1+(a_jU'+t)^2}\mathrm dt\\
    &=:& I_1+I_2.
  \end{array}
  \]
By a direct computation,
 \[
 \begin{array}{lllll}
   I_2&=&\displaystyle \sum_{j=2}^n\int^{\theta_jV}_{JV}\dfrac{1}{1+(a_jU'+t)^2}\mathrm dt
   \\
   &\leq & \displaystyle
   \sum_{j=2,\cdots,n\atop \theta_j\geq J}\int^{\theta_jV}_{JV}\dfrac{1}{1+(a_1U'+t)^2}\mathrm dt\\
   &\leq & \displaystyle \overline \Psi(U',sU'')\sum_{j=2,\cdots,n\atop \theta_j\geq J}\int_{JV}^{\theta_jV}\dfrac{1}{(\sqrt{1+(a_1U')^2}+t)^2}\mathrm dt\\
   &=&\displaystyle \overline \Psi(U',sU'')\sum_{j=2,\cdots,n\atop \theta_j\geq J} \dfrac{(\theta_j-J)V}{(\sqrt{1+(a_1U')^2}+JV)(\sqrt{1+(a_1U')^2}+\theta_jV)},\\
 \end{array}
 \]
 where $\overline \Psi(U',sU'')\geq 1$ is from the following estimate,
 \[
 \begin{array}{llll}
   \displaystyle \dfrac{(\sqrt{1+(a_1U')^2}+t)^2}{1+(a_1U'+t)^2} & = &
   \dfrac{t^2+2\sqrt{1+(a_1U')^2}t+1+(a_1U')^2}{t^2+2a_1U't+1+(a_1U')^2}\\
   &=& 1+\dfrac{2(\sqrt{1+(a_1U')^2}-a_1U')t}{t^2+2a_1U't+1+(a_1U')^2}\\
   &\leq & \displaystyle 1+\dfrac{2t}{(1+(a_1U')^2)\cdot (\sqrt{1+(a_1U')^2}+a_1U')}\\
   &\leq  & 1+\dfrac{4a_nsU''}{(1+(a_1U')^2)\cdot a_1U'}=:\overline\Psi (U',sU'').\\
 \end{array}
 \]
 Estimate \eqref{equ-temp-17} is used in the last inequality above.
 We thus get
 \[
 \begin{array}{lllll}
 \displaystyle
 I_2 &\leq&\displaystyle\overline \Psi(U',sU'')\cdot \dfrac{(\theta_2+\cdots+\theta_n)V}{(\sqrt{1+(a_1U')^2}+JV)(\sqrt{1+(a_1U')^2}+\theta_jV)}\\
 &\leq & \displaystyle
 \overline \Psi(U',sU'')\cdot \dfrac{(\theta_2+\cdots+\theta_n)V}{1+(a_1U')^2}.\\
 \end{array}
 \]
 Similarly, noticing the fact that for all $t\geq 0$, $1+(a_1U'+t)^2\leq (\sqrt {1+(a_1U')^2}+t)^2$ and hence
 \[
 \begin{array}{lllll}
   I_1&=& \displaystyle -\int_{\theta_1V}^{(1+J)V}\dfrac{1}{1+(a_1U'+t)^2}\mathrm dt\\
   &\leq &
   \displaystyle
   - \int_{\theta_1V}^{(1+J)V}\dfrac{1}{(\sqrt {1+(a_1U')^2}+t)^2}\mathrm dt\\
   &=& \displaystyle - \dfrac{(1+J-\theta_1)V}{(\sqrt {1+(a_1U')^2}+\theta_1V)(\sqrt {1+(a_1U')^2}+(1+J)V)}.\\
 \end{array}
 \]
Since we only need to prove for the case of $J\leq 1$,
 we continue to have
 \[
 \begin{array}{llll}
   I_1   &\leq & \displaystyle - \dfrac{(1+J-\theta_1)V}{(\sqrt {1+(a_1U')^2}+2V)^2}.\\
   & \leq &
   \displaystyle - \dfrac{(1+J-\theta_1)V}{(\sqrt {1+(a_1U')^2}+4a_nsU'')^2}\\
   &=& \displaystyle \underline\Psi (U',sU'')\cdot \dfrac{(\theta_1-(1+J))V}{1+(a_1U')^2},
 \end{array}
 \]
 where
 \[
 \underline\Psi(U',sU''):=\dfrac{1+(a_1U')^2}{(\sqrt{1+(a_1U')^2}+4a_nsU'')^2}.
 \]
 Especially, $\underline \Psi\leq 1<\overline\Psi$.
 By a direct computation and the second equality in \eqref{equ-temp-15}, we have
 \[
 \begin{array}{llll}
   F(D^2u)-f(\widehat a_J)&=& I_1+I_2\\
   &\leq & \displaystyle \dfrac{V}{1+(a_1U')^2}
   \left(
   \underline \Psi(\theta_1-(1+J))+\overline \Psi(\theta_2+\cdots+\theta_n)
   \right)\\
   &\leq &  \displaystyle \dfrac{V}{(a_1U')^2}
   \left(-(1+J)
   \underline \Psi+\overline \Psi
   \right).\\
 \end{array}
 \]
 Thus
 \[
 F(D^2u)\leq f(\widehat a_J)\quad\text{as long as}\quad J\geq \frac{\overline\Psi}{\underline\Psi}-1.
 \]
 This is provided by the following computation and the choice of $J$ in \eqref{equ-temp-16},
 \[
 \begin{array}{llll}
 \dfrac{\overline \Psi(U',sU'')}{\underline \Psi(U',sU'')}-1
 &=&  \displaystyle \dfrac{(\sqrt{1+(a_1U')^2}+4a_nsU'')^2\cdot (a_1U'+(a_1U')^3+4a_nsU'')}{(1+(a_1U')^2)^2\cdot a_1U'}-1\\
 &\leq & J(U',(s+1)U'').\\
 \end{array}
 \]
 The desired result $F(D^2u)\leq f(\bar a_J)$ follows from the fact that \eqref{equ-temp-17} and $
 f(\widehat a_J)\leq f(\bar a_J).$ The monotonicity and regularity of $J$ can be proved by the choice of $J(w,H)$ as in \eqref{equ-temp-16}.
\end{proof}

Next, to obtain an ODE satisfied by $U'$, we apply the implicit function theorem as in Lemma \ref{lem-ImplicitFuncResult}, but with additional $J$ terms.
Similar to the definition of $\underline w$ in previous section, by the monotonicity of $\arctan$ function, there exists a unique  increasing positive function $\overline w(s)$ defined on $[0,+\infty)$ determined by
\begin{equation}\label{equ-def-baru}
f(a_1\overline w(s),\cdots,a_n\overline w(s))=\underline g(s),
\end{equation}
where $\underline g$  is as in \eqref{equ-temp-3}.
\begin{lemma}\label{lem-ImplicitFuncResult-supersol}
  There exists a unique nonnegative smooth function $H(s,w)$ satisfying
  \[
  f(a_1w+2a_n(1+J(w,H))H,a_2w+2a_nJ(w,H)H,\cdots,a_nw+2a_nJ(w,H)H)=\underline g(s),
  \]
  in the set
  \[
  \left\{(s,w)~|~s\geq 0,~0<w\leq \overline w(s)\right\}.
  \]
  Especially, $H(s,w)$ is   monotone increasing with respect to $s$,
  \begin{equation}\label{equ-ZeroBarrier}
  H(s,w)=0\quad\text{if and only if}\quad w=\overline w(s)\quad\text{and}\quad
  \dfrac{\partial H}{\partial w}(\infty,1)=-M(A).
  \end{equation}
  Furthermore, there exists $C>0$ such that
  \begin{equation}\label{equ-Property-H-super}
  |H(s,w)-H(\infty,w)|\leq Cs^{-\frac{\beta}{2}},\quad\forall~s>0.
  \end{equation}
\end{lemma}
\begin{proof}
  On the one hand, for all $s\geq 0$ and  $0<w\leq\overline w(s)$,
  \[
  \begin{array}{lll}
  &\displaystyle \lim_{H\rightarrow 0}f(a_1w+2a_n(1+J)H,a_2w+2a_nJH,\cdots,a_nw+2a_nJH)\\
  =&\displaystyle
  f(a_1w,a_2w,\cdots,a_nw)\\
  \leq &\underline g(s).
  \end{array}
  \]
  On the other hand, from the definition of $J$ and $\overline w(s)$,
  \[
  \epsilon:=\inf\{J(w,1)~|~0<w\leq\overline w(0)\}>0.
  \]
  Hence by the monotonicity of $J(w,H)$ with respect to $H$, for all $s\geq 0$ and $0<w\leq \overline w(s)$,
  \begin{equation}\label{equ-temp-18}
  \begin{array}{lllll}
  &\displaystyle \lim_{H\rightarrow+\infty}f(a_1w+2a_n(1+J)H,a_2w+2a_nJH,\cdots,a_nw+2a_nJH)\\
  >&\displaystyle \lim_{H\rightarrow+\infty}f(a_1w+2a_n(1+\epsilon)H,a_2w+2a_n\epsilon H,\cdots,a_nw+2a_n\epsilon H)\\
  =&f(\infty,\infty,\cdots,\infty)\\
  =&
  \frac{n\pi}{2}>\underline g(s).
  \end{array}
  \end{equation}
  Thus  by the mean value theorem and the implicit function theorem,  there exists a unique $H(s,w)\geq 0$ such that the equality holds and it is a smooth, bounded function with respect to $s$ and $w$.

  Furthermore,  taking partial derivative with respect to $s$, we  have
  \[
  \begin{array}{lllll}
  0<\underline g'(s)&=&\displaystyle
  \dfrac{\partial f}{\partial\lambda_1}(\tilde\lambda)\cdot 2a_n \left(1+J(w,H)+H\dfrac{\partial J}{\partial H}(w,H)\right)\cdot \dfrac{\partial H}{\partial s}\\
  &&\displaystyle
  +\sum_{i=2}^n \dfrac{\partial f}{\partial \lambda_i}(\tilde\lambda)
  \cdot 2a_n\left(J(w,H)+H\dfrac{\partial J}{\partial H}(w,H)\right)\cdot  \dfrac{\partial H}{\partial s},\\
  \end{array}
  \]
  where
  \[
  \tilde \lambda:=\left(a_1w+2a_n(1+J(w,H))H,a_2w+2a_nJ(w,H)H,\cdots,a_nw+2a_nJ(w,H)H\right).
  \]
  Hence $H(s,w)$ is monotone increasing with respect to $s$.

  By the monotonicity of $\arctan$ and \eqref{equ-def-baru}, when $w(s)=\overline w(s)$, we have the following two equalities
  \[
  \begin{array}{llll}
    \underline g&=& f(a_1\overline w,a_2\overline w,\cdots,a_n\overline w),\\
    \underline g&=& f(a_1\overline w+2a_n(1+J(\overline w,H))H,a_2\overline w+2a_nJ(\overline w,H)H,\cdots,a_n\overline w+2a_nJ(\overline w,H)H).
  \end{array}
  \]
  Since $2a_n(1+J(\overline w,H))H\geq 0$ and $2a_nJ(\overline w,H)H\geq 0$, the only possibility to make the equalities above hold is $H(s,\overline w)=0$. It is clearly, $H(s,w)>0$ when $w<\overline w(s)$.

  Next, we compute the partial derivative of $H$ with respect to $w$ at $(\infty,1)$.
  By taking partial derivative with respect to $w$, we have
  \[
  \begin{array}{lllll}
    0&=&\displaystyle \dfrac{\partial f}{\partial\lambda_1}(\tilde\lambda)\cdot \left(a_1+2a_n\left(1+J(w,H)\right)\dfrac{\partial H}{\partial w}+2a_nH\left(\dfrac{\partial J}{\partial w}+\dfrac{\partial J}{\partial H}\cdot \dfrac{\partial H}{\partial w}\right)\right)\\
    &&\displaystyle+\sum_{i=2}^n\dfrac{\partial f}{\partial\lambda_i}(\tilde\lambda)
    \cdot \left(a_i+2a_nJ(w,H)\dfrac{\partial H}{\partial w}+2a_nH\left(\dfrac{\partial J}{\partial w}+\dfrac{\partial J}{\partial H}\cdot \dfrac{\partial H}{\partial w}\right)\right).\\
  \end{array}
  \]
  From the definition  \eqref{equ-temp-16}  of $J$ in Lemma \ref{lem-supsolu-estimate-refine}, we have
  \[
  \lim_{(s,w)\rightarrow (\infty,1)}H(s,w)=0,
  \lim_{(s,w)\rightarrow (\infty,1)}J(w,H)=0,
  \quad\text{and}\quad
  \lim_{(s,w)\rightarrow (\infty,1)}\tilde\lambda=(a_1,a_2,\cdots,a_n).
  \]
  Consequently, by the smoothness of $J$ and $H$, sending $(s,w)$ to $(\infty,1)$, we have
  \[
  0=
    \dfrac{1}{1+a_1^2}\cdot\left(a_1+2a_n\dfrac{\partial H}{\partial w}(\infty,1)\right)+\sum_{i=2}^n\dfrac{a_i}{1+a_i^2},
  \]
  i.e.,
  \[
  \dfrac{\partial H}{\partial w}(\infty,1)=-\dfrac{1+a_1^2}{2a_n}\sum_{i=1}^n\dfrac{a_i}{1+a_i^2}=-M(A).
  \]
  This finishes the proof of \eqref{equ-ZeroBarrier}.

  It remains to prove \eqref{equ-Property-H-super}. Since $H(s,w)$ is bounded, monotone increasing with respect to $s$ and $J(w,H)$ is bounded, monotone increasing with respect to $H$, we have
  \[
  \begin{array}{lllll}
    & (1+J(w,H(\infty,w)))H(\infty,w)-
    (1+J(w,H(s,w)))H(s,w)\\
  =& (1+J(w,H(\infty,w)))H(\infty,w)-
    (1+J(w,H(s,w)))H(\infty,w)\\
    &+\left((1+J(w,H(s,w)))H(\infty,w)-
    (1+J(w,H(s,w)))H(s,w)\right)\\
  \geq & H(\infty,w)-H(s,w),
\end{array}
  \]
  and  there exists $C>0$ such that $
  2a_n(1+J)H\leq C$.
  Hence by the Newton--Leibnitiz formula and the monotonicity,
  \[
  \begin{array}{lllll}
  & g(\infty)-\underline g(s)\\
  \geq & \displaystyle
     \arctan (a_1w+2a_n(1+J(w,H(\infty,w)))H(\infty,w)) \\
    &\displaystyle  -
    \arctan (a_1w+2a_n(1+J(w,H(s,w)))H(s,w))
     \\
    \geq & \displaystyle \dfrac{2a_n}{1+(a_1+C)^2}\left((1+J(w,H(\infty,w)))H(\infty,w)-
    (1+J(w,H(s,w)))H(s,w)
    \right)\\
    \geq & \dfrac{2a_n}{1+(a_1+C)^2}(H(\infty,w)-H(s,w)).
  \end{array}
  \]
  This finishes the proof of this lemma.
\end{proof}

To proceed, we prove that by solving an ordinary differential equation, there exists a supersolution to \eqref{equ-Lagrangian} in $\mathbb R^n\setminus\{0\}$.
\begin{corollary}\label{Coro-temp2}
  Let $u$, $U$, $A$, and $H$ be as in Lemmas \ref{lem-supsolu-estimate}, \ref{lem-supsolu-estimate-refine}, and \ref{lem-ImplicitFuncResult-supersol}. If
  \[
  0<U'(s)\leq\overline w(s),\quad U''(s)\geq 0\quad\text{and}\quad U''(s)=\dfrac{H(s,U'(s))}{s+1}\quad\text{in }s>0,
  \]
  then $u$ is a supersolution to \eqref{equ-Lagrangian} in $\mathbb R^n\setminus\{0\}$.
\end{corollary}
\begin{proof}
  By the equation satisfied by $U''$, we have
  \[
  sU''\leq(s+1)U''=H(s,U'),\quad\forall~s>0.
  \]
  Hence by Lemmas \ref{lem-supsolu-estimate-refine} and \ref{lem-ImplicitFuncResult-supersol},
  \[
  \begin{array}{llll}
    F(D^2u)&\leq & f(a_1U'+2a_n(1+J)sU'',a_2U'
    +2a_nJsU'',\cdots,a_nU'+2a_nJsU'')\\
    &\leq & f(a_1U'+2a_n(1+J)H,a_2U'+2a_nJH,\cdots,a_nU'+2a_nJH)\\
    &=& \underline g(s),\\
  \end{array}
  \]
  for all $s>0$.
  This proves $u$ is a generalized symmetric supersolution to \eqref{equ-Lagrangian} in $\mathbb R^n\setminus\{0\}$.
\end{proof}

Similar to the proof of Lemmas \ref{lem-subsolution-Existence} and \ref{lem-ODEsolution-subsol},  we have  the following existence result of  initial value problem of ODE.

\begin{lemma}\label{lem-ODEsolution-supersol}
  For any $w_0\in (0,\overline w(0))$, there exists a unique solution $W$ to
  \begin{equation}\label{equ-system-Dirich-supersol}
  \left\{
  \begin{array}{llllll}
    \displaystyle \dfrac{\mathrm{d} w}{\mathrm{d} s}=\dfrac{H(s,w)}{s+1}, & \text{in } s>0,\\
    w(0)=w_0.\\
  \end{array}
  \right.
  \end{equation}
  Furthermore,
  \[
  W'(s)\geq 0,\quad \overline w(s)> W(s)>w_0,\quad\forall~s\in(0,+\infty),
  \]
  and $W(s)$ has the  asymptotic behavior at infinity as in \eqref{equ-temp-14}.
\end{lemma}
\begin{proof}
  Since $0<w_0<\overline w(0)$, there exist $s_0,\epsilon>0$ such that
  \[
  [0,s_0]\times\left[w_0-\epsilon,w_0+\epsilon\right]\subset\{(s,w)~|~s\geq 0, 0<w\leq \overline w(s)\}.
  \]
  Since $H(s,w)$ is a nonnegative smooth function, the right hand side term $\frac{H(s,w)}{s+1}$ is Lipschitz in the rectangle $[0,s_0]\times\left[
  w_0-\epsilon,w_0+\epsilon
  \right]$. By the Picard--Lindel\"of theorem, the initial value problem \eqref{equ-system-Dirich-supersol} admits locally a unique solution $W$ near $s=0$ and we shall prove that the solution can be extended to $s\in[0,+\infty)$.

  By \eqref{equ-ZeroBarrier} and the fact that $H$ is nonnegative, following the argument as in the proof of Lemma \ref{lem-ODEsolution-subsol}, we can prove that $W(s)$ cannot touch $\overline w(s)$ from below. Consequently, $W(s)$ is monotone increasing and $w_0\leq W(s)<\overline w(s)$ holds as long as $W(s)$ exists. By the Carath\'eodory extension theorem of ODE, the solution $W$ exists on entire $[0,+\infty).$

  Now we prove the asymptotic behavior of $W$ near infinity. Since $W$ is positive, monotone increasing and bounded from above by $\overline w(s)\leq 1$, $W$ admits a finite limit $0<W(\infty)\leq 1$ at infinity.
  Especially, we claim that $\liminf_{s\rightarrow\infty}H(s,W(s))=0$, which implies that there exists a subsequence $\{s_j\}_{j=1}^{\infty}$ such that
  \[
  \lim_{j\rightarrow \infty}s_j=+\infty\quad\text{and}\quad
  \lim_{j\rightarrow \infty}H(s_j,W(s_j))=0.
  \]
  Arguing by contradiction, we suppose  $2\epsilon:=\liminf_{s\rightarrow\infty}H(s,W(s))>0$. Then there exists $\bar s$ sufficiently large such that
  \[
  \dfrac{H(s,W(s))}{s+1}>\dfrac{\epsilon}{s+1},\quad\forall~s\geq\bar s.
  \]
  This contradicts to the Newton--Leibnitiz formula
  \[
  W(\infty)-W(\bar s)=\int_{\bar s}^{+\infty}\dfrac{H(r,W(r))}{r+1}\mathrm dr>\epsilon\int_{\bar s}^{+\infty}\dfrac{1}{r+1}\mathrm dr=+\infty,
  \]
  and proves the existence of the subsequence $\{s_j\}$.
  Furthermore, we claim that $W(\infty)=1$. Arguing by contradiction, if $W(\infty)<1$, then by  Lemma \ref{lem-ImplicitFuncResult-supersol} and sending $j\rightarrow\infty$,
  \[
  \begin{array}{lllll}
    g(\infty)&=&\displaystyle \lim_{j\rightarrow\infty}\underline g(s_j)\\
    &=&\displaystyle \arctan (a_1W(\infty)+2a_n(1+J(W(\infty),0))\cdot 0)\\
    &&\displaystyle+\sum_{i=2}^n\arctan (a_iW(\infty)+2a_nJ(W(\infty),0)\cdot 0)\\
    &=&\displaystyle \sum_{i=1}^n\arctan(a_iW(\infty))\\
    &<& g(\infty),
  \end{array}
  \]
  which is a contradiction. Next, we refine the asymptotic behavior by setting
  \[
  t:=\ln (s+1)\in (0,+\infty)\quad\text{and}\quad\varphi(t):=W(s(t))-1.
  \]
  By a direct computation, for all $t\in (0,+\infty)$,
  \[
  \begin{array}{llll}
    \varphi'(t)&=& W'(s(t))\cdot e^t\\
    &=& H(s(t),\varphi+1)\\
    &=& \underbrace{H(s(t),\varphi+1)-H(\infty,\varphi+1)}_{=:H_1(t,\varphi)}
    +\underbrace{H(\infty,\varphi+1)}_{=:H_2(\varphi)}.\\
  \end{array}
  \]
  Furthermore, by property \eqref{equ-Property-H-super} of $H(s,w)$ in Lemma \ref{lem-ImplicitFuncResult-supersol},   there exist $C>0$ such that for all $t\geq 1, |\varphi|\ll 1$,
  \[
  |H_1(t,\varphi)|\leq Ce^{-\frac{\beta}{2}t},\quad\text{and}\quad
  \left|H_2(\varphi)+M(A)\varphi\right|\leq C\varphi^2.
  \]
  Then the asymptotic behavior of $W$ follows from the asymptotic stability theory of ODE, which is identical to the proof in Lemma \ref{lem-ODEsolution-subsol}. This finishes the proof of this lemma.
\end{proof}
\begin{proof}[Proof of Lemma \ref{lem-supersolution-Existence}]
Take  any $0<w_0<\overline w(0)$ and $W$ to be the solution to \eqref{equ-system-Dirich-supersol} from Lemma \ref{lem-ODEsolution-supersol}. Let
 \[
 \overline u(x):=\overline U(s(x)):=\int_0^{s}W(r)\mathrm dr+C,
 \]
 where $C$ is a constant to be determined.

 By Corollary \ref{Coro-temp2},  $\overline u$ is a supersolution in $\mathbb R^n\setminus\{0\}$. By Lemma \ref{lem-ODEsolution-supersol} and the condition that $M(A)>1$ and $\beta>2$, we may choose
 \[
 C:=\lim_{s\rightarrow\infty}\int_0^{s}(1-W(r))\mathrm dr.
 \]
 This makes $\overline u$ satisfies the desired asymptotic behavior at infinity. Furthermore, we also have
 \[
 D\overline u(x)=\overline U'(s(x))\cdot Ax\rightarrow 0,\quad\text{as }x\rightarrow 0.
 \]
 Hence by Theorem 1.1 in \cite{Caffarelli-Li-Nirenberg--ViscositySol-III}, the function $\overline u$ is a supersolution to \eqref{equ-Lagrangian} in $\mathbb R^n$ in viscosity sense.
\end{proof}

\section{Proof of Theorem \ref{thm-n>=3}}\label{sec-ProofThm1}

Since $g>\frac{(n-2)\pi}{2}$ and $g\in C^{m}(\mathbb R^n)$ for some $m\geq 2$, by the existence result as in  Bhattacharya \cite{Bhattacharya-Dirichlet-LagMeanCurva} and
Lu \cite{Lu-Dirichlet-Lagrangian}, there exists a unique classical solution $u_s$  that solves
\begin{equation}\label{equ-Dirichlet-Existence}
\left\{
\begin{array}{llll}
  F(D^2u_s)=g(x), & \text{in }D_s,\\
  u_s=s, & \text{on }\partial D_s.\\
\end{array}
\right.
\end{equation}
Hereinafter, $D_s$ denotes the  ellipsoid $
D_s:=\left\{x\in\mathbb R^n~|~s(x)=\frac{1}{2}x^TAx\leq s\right\}.
$
Especially, $u_s\in C^{m+1,\alpha}(\overline D_s)$ for any $0<\alpha<1$. Now we provide a uniform bounds of $u_s$ as below.

\begin{lemma}\label{Lem-UniformBds-C0} Let $u_s$ be the solutions as above.
  There exists a positive constant $C_1$  independent of $s$ such that
  \[
  \left|u_s(x)-\frac{1}{2}x^TAx\right|\leq C_1\quad\text{in }D_s.
  \]
\end{lemma}
\begin{proof}
  Let $\underline u$ and $\overline u$ be the subsolution and supersolution to \eqref{equ-Lagrangian} in $\mathbb R^n$ as in Lemmas \ref{lem-subsolution-Existence} and \ref{lem-supersolution-Existence} respectively.
  Then
  \[
  \beta_-:=\inf_{\mathbb R^n}\left(\frac{1}{2}x^TAx-\underline u(x)\right)\quad\text{and}\quad
  \beta_+:=\sup_{\mathbb R^n}\left(\frac{1}{2}x^TAx-\overline u(x)\right)
  \]
  satisfies
  \[
  -\infty<\beta_-<0<\beta_+<+\infty
  \]
  and consequently
  \[
  \underline u(x)+\beta_-\leq \frac{1}{2}x^TAx\leq\overline u(x)+\beta_+\quad\text{in }\mathbb R^n.
  \]
  Since $\underline u+\beta_-$ and $\overline u+\beta_+$  are viscosity subsolution and supersolution to \eqref{equ-Lagrangian} in $\mathbb R^n$ respectively, by comparison principle as Theorem 3.3 in \cite{User'sGuide-ViscositySol} we have
  \[
  \underline u(x)+\beta_-\leq u_s(x)\leq\overline u(x)+\beta_+\quad\text{in }D_s
  \]
  for any $s>0$.
  This proves the desired result.
\end{proof}

The following gradient and Hessian estimates  by  Bhattacharya--Monney--Shankar \cite{Bhattacharya-Monney-Shankar-GradientEsti-Critical} will be used to provide locally uniform bounds of $u_s$ in $C^{m+1,\alpha}_{loc}(\mathbb R^n)$ with $0<\alpha<1$.
\begin{theorem}\label{thm-GradientEstimate}
  Let  $u$ be a $C^4$ solution to \eqref{equ-Lagrangian}
  in $B_{R}$,
  where $g\in C^{2}(B_{R})$ satisfies $|g|\geq \frac{(n-2)\pi}{2}$. Then
  \[
  |Du(0)|\leq \frac{C}{R}\left(1+R^4(\underset{B_R}{\mathtt{osc}}u)^2\right)
  \]
  and
  \[
  |D^2u(0)|\leq C\exp\left(\dfrac{C}{R^{2n-2}}\max_{B_R}|Du|^{2n-2}\right),
  \]
  where hereinafter $B_r(x)$ denote the ball centered at $x$ with radius $r$, $B_r:=B_r(0)$, and the constant  $C$ relies only on $n$ and $||g||_{C^2(B_R)}$.
\end{theorem}

\begin{lemma}\label{lem-uniformBds-Cm}
  For any $0<\alpha<1$ and any bounded subset $K\subset\mathbb R^n$, there exists $C>0$ such that
  \[
  ||u_s||_{C^{m+1,\alpha}(K)}\leq C,\quad\forall~s>0.
  \]
\end{lemma}
\begin{proof}
  By Lemma \ref{Lem-UniformBds-C0}, we have proved that $u_s$ are locally uniformly bounded in $C^0(\mathbb R^n)$. It remains to  obtain estimates for higher order derivatives.

  For any bounded smooth set $K\subset\mathbb R^n$, we take a sufficiently large $S$ such that $K\subset\subset D_S$ and set $R:=\frac{1}{4}\mathrm{dist}(K,\partial D_S)$. By Lemma \ref{Lem-UniformBds-C0},
  \[
  \left|u_s(x)-\frac{1}{2}x^TAx\right|\leq C_1\quad\text{in } D_S,\quad\forall~s>S.
  \]
  Especially, we have uniform $C^0$-bound
  \[
  |u_s(x)|\leq S+C_1\quad\text{in }D_S,\quad\forall~s>S.
  \]
  For any $x$ satisfying $\mathrm{dist}(x,K)<2R$, we apply the first inequality in  Theorem \ref{thm-GradientEstimate} in $B_{2R}(x)\subset D_S$ and obtain
  \[
  |Du_s(x)|\leq \dfrac{C}{R}\left(1+4R^4(S+C_1)^2\right),\quad\forall~s>S,
  \]
  where $C$ relies only on $||g||_{C^1(D_S)}$ and $n$. This proves the locally uniform $C^1$-bounds of $u_s$.

  Then for any $x\in K$, we apply the second inequality in  Theorem \ref{thm-GradientEstimate} in $B_R(x)$ and obtain
  \[
  |D^2u_s(x)| \leq C\exp\left(\dfrac{C^{2n-1}}{R^{2n-2}}
  \left(1+4R^4(S+C_1)^2\right)^{2n-2}\right),\quad\forall~s>S.
  \]
  This proves the locally uniform $C^2$-bounds of $u_s$.

  Especially, $F$ is uniformly elliptic with respect to $u_s$ in $K$. Furthermore, since $g>\frac{(n-2)\pi}{2}$, Yuan \cite{Yuan-GlobalSolution-SPL} proved that $F$ is a concave operator to $u$ in this case. Hence for any $3\leq k\leq m+1$,  $0<\alpha<1$ and $K'\subset\subset K$, by the Evans--Krylov estimates and the Schauder theory as Theorems 6.2 and 17.14 in \cite{Book-Gilbarg-Trudinger},
  \[
  ||u_s||_{C^{k,\alpha}(K')}\leq C
  \]
  for some constant $C$ independent of $s$.
\end{proof}

The estimates above enable us to obtain a limit function $u_{\infty}$, which is a classical solution to \eqref{equ-Lagrangian} in $\mathbb R^n$. Still, we need the following asymptotic behavior result to prove Theorem \ref{thm-n>=3}.

\begin{lemma}\label{lem-AsymptoticBehavior}
  Let $n\geq 3$, $u$ be a smooth solution to \eqref{equ-Lagrangian} in $\mathbb R^n$ and $g\in C^{m}(\mathbb R^n)$ satisfy
  \begin{equation}\label{equ-RHScond-temp}
  |g|\geq \frac{(n-2)\pi}{2}\quad\text{and}\quad
  g(x)=g(\infty)+O_m(|x|^{-\beta})\quad\text{as }|x|\rightarrow\infty,
  \end{equation}
  where $m\geq 2$ and $\beta>2$. If there exists a positive symmetric matrix $A$ with $F(A)=g(\infty)$ such that
  \begin{equation}\label{equ-VanishingCondition-Rough}
  u(x)-\frac{1}{2}x^TAx=o(|x|^2)\quad\text{as }|x|\rightarrow\infty.
  \end{equation}
  Then there exist $b\in\mathbb R^n$  and $c\in\mathbb R$ such that
  \eqref{equ-system-Dirich} holds.
  \begin{equation}\label{equ-asymptotic-Refine}
  u-\left(\frac{1}{2}x^TAx+b x+c\right)=\left\{
  \begin{array}{lllll}
    O_{m+1}(|x|^{2-\min\{\beta,n\}}), & \text{if }\beta\neq n,\\
    O_{m+1}(|x|^{2-n}(\ln|x|)), & \text{if }\beta=n.\\
  \end{array}
  \right.
  \end{equation}
\end{lemma}
\begin{proof}
  Let
  \[
  w(x):=u(x)-\frac{1}{2}x^TAx.
  \]
  For sufficiently large $R:=|x|>1$, set
  \[
  u_R(y):=\left(\dfrac{4}{R}\right)^2u\left(x+\frac{R}{4}y\right)\quad\text{and}\quad
  w_R(y):=\left(\dfrac{4}{R}\right)^2w\left(x+\frac{R}{4}y\right)
  \quad\text{in }B_2.
  \]
  Then
  \[
  F(D^2u_R(y))=F\left(D^2u\left(x+\frac{R}{4}y\right)\right)=g\left(x+\frac{R}{4}y\right)=:g_R(y).
  \]
  By condition \eqref{equ-VanishingCondition-Rough},
  \[
  \max_{y\in \overline B_2}|u_R(y)|\leq \dfrac{16}{R^2}\max_{z\in \overline{B_{\frac{3R}{2}}\setminus B_{\frac{R}{2}}}}|u(z)|\leq C
  \]
  for some $C>0$ independent of $R$ and
  \[
  ||w_R||_{L^{\infty}(B_2)}\leq \dfrac{16}{R^2}\left\Vert
  u-\frac{1}{2}x^TAx
  \right\Vert_{L^{\infty}(B_{\frac{3R}{2}}\setminus B_{\frac{R}{2}})}=o(1)
  \]
  as $R\rightarrow\infty$.
  By a direct computation and condition \eqref{equ-RHScond-temp},
  \[
  ||g_R-g(\infty)||_{C^{m}(B_2)}\leq CR^{-\beta}
  \]
  for some positive constant $C$ independent of $R$.
  Hence by the gradient estimate and Hessian estimate as in Theorem \ref{thm-GradientEstimate}, we have $C>0$ independent of $R$ such that
  \[
  ||u_R||_{C^2(B_1)}\leq C.
  \]
  Consequently $F$ is uniformly elliptic with respect to all $u_R$ and concave in level set sense \cite{Yuan-GlobalSolution-SPL}. By the Evans--Krylov estimate and the Schauder theory again, for any $0<\alpha<1$, we have
  \begin{equation}\label{equ-regularity-temp}
  ||u_R||_{C^{2,\alpha}(B_{\frac{1}{2}})}\leq C\quad\text{and hence}\quad
  ||w_R||_{C^{2,\alpha}(B_{\frac{1}{2}})}\leq C
  \end{equation}
  for some $C>0$ independent of $R$. Notice that by the Newton--Leibnitz formula, $w_R$ satisfies
  \[
  a_{ij}^R(y)D_{ij}w_R(y)=g_R(y)-g(\infty)\quad\text{in }B_2,
  \]
  where
  \[
  a_{ij}^R(y)=\int_0^1D_{M_{ij}}F(A+tD^2w_R(y))\mathrm dt,
  \]
  and
  $D_{M_{ij}}F$ denotes partial derivative of $F(M)$ with respect to the $M_{ij}$-component.
  By  \eqref{equ-regularity-temp}, $w_R$ satisfies a uniformly elliptic equation with $C^{\alpha}$-regular coefficients. Thus by the Schauder theory again, there exists $C>0$ independent of $R$ such that
  \[
  |D^2w_R(0)|\leq C\left(
  ||w_R||_{L^{\infty}(B_1)}+||g_R-g(\infty)||_{C^{1}(B_1)}
  \right)=o(1)
  \]
  as $R\rightarrow\infty$. It follows that
  \[
  |D^2u(x)-A|=|D^2w(x)|=|D^2w_R(0)|=o(1),
  \]
  as $|x|\rightarrow\infty$.
  By Theorem 1.1 in \cite{Jia-Xiaobiao-AsymGeneralFully}, we have the desired result immediately.
\end{proof}

\begin{remark}
  In the original statement of Theorem 1.1 in \cite{Jia-Xiaobiao-AsymGeneralFully}, the asymptotic behavior results were stated as below
  \[
  u-\left(\frac{1}{2}x^TAx+b x+c\right)=\left\{
  \begin{array}{lllll}
    O_{m+1}(|x|^{2-\beta}), & \text{if }\beta<n,\\
    O_{m+1}(|x|^{2-s}),~\text{ for all }s\in (2,n), & \text{if }\beta=n,\\
    O_{m+1}(|x|^{2-n}), & \text{if }\beta>n.\\
  \end{array}
  \right.
  \]
  However this can be improved further into \eqref{equ-system-Dirich} by replacing the barrier function $|x|^{2-\frac{s+n}{2}}$ when $\beta=n$ in Step 2.2 of \cite{Jia-Xiaobiao-AsymGeneralFully} into $|x|^{2-n}(\ln|x|)$. See also the discussions and results in \cite{Liu-Bao-2021-Expansion-LagMeanC} etc.
\end{remark}

\begin{proof}[Proof of Theorem \ref{thm-n>=3}]
Let $u_s$ be the solution to Dirichlet problem \eqref{equ-Dirichlet-Existence} as in Lemma \ref{Lem-UniformBds-C0}. Then by Lemma
\ref{lem-uniformBds-Cm}, the Arzela--Ascoli theorem and the diagonal argument, there exist a subsequence $\{s_i\}_{i=1}^{\infty}$ and $u_{\infty}\in C^{m+1}_{loc}(\mathbb R^n)$ such that
\[
s_i\rightarrow\infty\quad\text{and}\quad
u_{s_i}\rightarrow u_{\infty}\quad\text{in }C_{loc}^{m+1}(\mathbb R^n)\quad\text{as }i\rightarrow\infty.
\]
In particular, $u_{\infty}$ is a classical solution to \eqref{equ-Lagrangian} in $\mathbb R^n$ and satisfies
\[
u_{\infty}(x)-\frac{1}{2}x^TAx=O(1)\quad\text{as }|x|\rightarrow\infty.
\]

Applying Lemma \ref{lem-AsymptoticBehavior} to $u_{\infty}$ obtained above, we have $c_{\infty}\in\mathbb R$ such that \[
u:=
u_{\infty}-c_{\infty}
=\frac{1}{2}x^TAx+\left\{
  \begin{array}{lllll}
    O_{m+1}(|x|^{2-\min\{\beta,n\}}), & \text{if }\beta\neq n,\\
    O_{m+1}(|x|^{2-n}(\ln|x|)), & \text{if }\beta=n,\\
  \end{array}
  \right.
\]
as $|x|\rightarrow\infty$. Then $u$ is an entire solution to  equation \eqref{equ-Lagrangian} satisfying  prescribed asymptotic behavior \eqref{equ-system-Dirich} at infinity with $b=0$ and $c=0$, and the regularity of $u$ follows from the Schauder theory. The uniqueness of $u$ follows from maximum principle as Theorem 17.1 in \cite{Book-Gilbarg-Trudinger}.
\end{proof}

\section{Proof of Theorem \ref{thm-Optimality}}\label{sec-ProofThm2}

In this section, we prove Theorem \ref{thm-Optimality}
by analyzing radially symmetric solutions and
arguing by contradiction.
Firstly, we consider radially symmetric $g$ and analyze the asymptotic behavior at infinity of radially symmetric solutions.
Secondly, by taking $A=\tan \frac{g(\infty)}{n} I$ and assuming there exists a classical solution $u$ to \eqref{equ-Lagrangian} satisfying asymptotic behavior \eqref{equ-ContraAssum}, we prove that $u$ is indeed a radially symmetric solution but the asymptotic behavior of  $u$ contradicts to the result just mentioned.

For any  $0<\beta\leq 2$ and $\frac{(n-2)\pi}{2}<G(\infty)<G(0)<\frac{n\pi}{2}$, let
\begin{equation}\label{equ-example-g}
G(r):=\left\{
\begin{array}{lllll}
  G(0), & \text{when }r\leq 1,\\
  G(\infty)+r^{-\beta}, & \text{when }r>2\max\left\{1, (G(0)-G(\infty))^{-\frac{1}{\beta}}\right\},\\
\end{array}
\right.
\end{equation}
be a smooth function on $[0,+\infty)$ that is monotone decreasing in $(1,+\infty)$.
Take $g\in C^{\infty}(\mathbb R^n)$ to be the radially symmetric function satisfying
\[
g(x)=G(r),\quad r:=|x|\quad\text{in }\mathbb R^n,
\]
which satisfies condition \eqref{equ-cond-g} for all $m\geq 0$ with
$g(0)=G(0)$
and $g(\infty)=G(\infty)$.

We start with classifying all radially symmetric classical solutions and analyzing their asymptotic behavior.
More explicitly, for  a radially symmetric function  $u(x)=U(r)$, we have
\[
\lambda(D^2u)=\left(U'',\frac{U'}{r},\cdots,\frac{U'}{r}\right).
\]
Hence if $u$
is a radially symmetric classical solution to \eqref{equ-Lagrangian} in $\mathbb R^n$, then  $W:=\frac{U'}{r}$ satisfies
\begin{equation}\label{equ-temp-20}
\arctan (W+rW')+(n-1)\arctan W=G(r)\quad \text{in }r>0.
\end{equation}

\begin{lemma}\label{lem-Def-h}
  There exists a unique  smooth function $h(r,w)$ satisfying
  \begin{equation}\label{equ-temp-5}
  \arctan(w+h(r,w))+(n-1)\arctan w=G(r)\quad\text{in }\left\{(r,w)~|~r\geq 0,~w>\tan\frac{G(r)-\frac{\pi}{2}}{n-1}\right\}.
  \end{equation}
  Especially,
  \begin{equation}\label{equ-Property-h-1}
  h(r,w)\equiv h(0,w)\quad\text{in }[0,1],\quad
  h\left(r,\tan\frac{G(r)}{n}\right)=0\quad\text{in }[0,+\infty),
  \end{equation}
  \begin{equation}\label{equ-Property-h-3}
  \dfrac{\partial h}{\partial w}\left(r,\tan\frac{G(r)}{n}\right)=-n\quad\text{in }[0,+\infty),
  \end{equation}
  and $h(r,w)$ is monotone decreasing with respect to $w$ and non-increasing with respect to $r$.
  Furthermore,  there exist $\delta, R, C>0$  such that
  \begin{equation}\label{equ-property-h-temp}
  r^{-\beta}\leq h(r,w)-h(\infty,w)\leq Cr^{-\beta}
  \end{equation}
  for all $(r,w)\in (R,+\infty)\times(\tan (\frac{G(\infty)}{n}-\delta),
  \tan(\frac{G(\infty)}{n}+\delta))$.
\end{lemma}

\begin{proof}
Notice that for all $r\geq 0$ and $w>\tan\frac{G(r)-\frac{\pi}{2}}{n-1}$, we have
\[
\lim_{h\rightarrow-\infty}\arctan(w+h)+(n-1)\arctan w<\frac{(n-2)\pi}{2}<G(r)
\]
and
\[
\lim_{h\rightarrow+\infty}\arctan(w+h)+(n-1)\arctan w
=\frac{\pi}{2}+(n-1)\arctan w>G(r).
\]
Hence by the mean value theorem and monotonicity of $\arctan$, there exists a unique function $h(r,w)$ such that \eqref{equ-temp-5} holds.
Especially, since
\[
G(r)\equiv G(0)\quad\text{in }[0,1]\quad\text{and}\quad
n\arctan \left(\tan\frac{G(r)}{n}\right)=G(r)\quad\text{in }[0,+\infty),
\]
two equalities in \eqref{equ-Property-h-1} follow from \eqref{equ-temp-5}.

  By the implicit function theorem,
  \[
  \dfrac{1}{1+(w+h)^2}\cdot \left(1+\frac{\partial h}{\partial w}\right)+\dfrac{n-1}{1+w^2}=0\quad\text{and}
  \quad
  \dfrac{1}{1+(w+h)^2}\cdot\dfrac{\partial h}{\partial r}=G'(r).
  \]
  Consequently, $h(r,w)$ is monotone decreasing with respect to $w$ and is monotone non-increasing with respect to $r$.
  Equality  \eqref{equ-Property-h-3} follows from the second equality in    \eqref{equ-Property-h-1} and the computation of partial derivative above.

  Eventually, we prove \eqref{equ-property-h-temp} by the Newton--Leibnitz formula.
  Choose sufficiently large $R$ and sufficiently small $\delta>0$ such that
  \[
  R>2\max\left\{1,(G(0)-G(\infty))^{-\frac{1}{\beta}}, \left(
  \frac{\pi}{2}-\frac{G(\infty)}{n}\right)^{-\frac{1}{\beta}}\right\}
  \]
  and
  \[
  \delta<\dfrac{1}{n-1}\min\left\{\frac{G(\infty)}{n},\frac{\pi}{2}-\frac{G(\infty)}{n}-R^{-\beta}\right\}.
  \]
  By   \eqref{equ-temp-5}, for all $(r,w)\in (R,+\infty)\times(\tan (\frac{G(\infty)}{n}-\delta),
  \tan(\frac{G(\infty)}{n}+\delta))$, we have
  \[
  \begin{array}{llll}
  w+h(r,w)&=&\tan\left(G(r)-(n-1)\arctan w\right)\\
  &>& \tan\left(G(\infty)-(n-1)\cdot\left(\frac{G(\infty)}{n}+\delta\right)\right)\\
  &>&0\\
  \end{array}
  \]
  and
  \[
  \begin{array}{llll}
    w+h(r,w) & =& \tan\left(G(r)-(n-1)\arctan w\right)\\
    &<& \tan\left(G(\infty)+r^{-\beta}-(n-1)\cdot\left(\frac{G(\infty)}{n}-\delta\right)\right)\\
    &<&\tan\left(\frac{G(\infty)}{n}+(n-1)\delta+R^{-\beta}\right).\\
  \end{array}
  \]
  Hence  for all $(r,w)$ in the range above,  $w+h(r,w)$
  is  bounded and
  there exists $C>0$ such that
  \[
  \begin{array}{llllll}
    r^{-\beta}&=&G(r)-G(\infty)\\
    &=& \arctan (w+h(r,w))-\arctan (w+h(\infty,w))\\
    &\geq&\frac{1}{C}\left(h(r,w)-h(\infty,w)\right),
  \end{array}
  \]
  and
  \[
  \begin{array}{llllll}
    r^{-\beta}&=&G(r)-G(\infty)\\
    &=& \arctan (w+h(r,w))-\arctan (w+h(\infty,w))\\
    &\leq&  h(r,w)-h(\infty,w).
  \end{array}
  \]
  This finishes the proof of this lemma.
\end{proof}

\begin{lemma}\label{lem-ODEsolution-radial}
Let $n\geq 3$, $0<\beta\leq 2$, $G(r)$ and $h(r,w)$ be the functions from Lemma \ref{lem-Def-h}.
  There exists a unique solution $W\in C^1([0,+\infty))$ to
  \begin{equation}\label{equ-temp-11}
  w'=\frac{h(r,w)}{r}\quad\text{in }r>0.
  \end{equation}
  Furthermore,
  \begin{equation}\label{equ-origin-quadratic}
  W'(r)\equiv 0\quad\text{and}\quad W(r)\equiv \tan\frac{G(0)}{n},\quad\forall~r\in[0,1],
  \end{equation}
  and there exist $C_1, C_2, R>0$ such that
  \[
  C_1 r^{-\beta}\leq  W(r)-\tan\frac{G(\infty)}{n} \leq C_2 r^{-\beta},\quad\forall~r>R.
  \]
\end{lemma}
\begin{proof}
Firstly, we prove the existence of solution. Since $h(r,w)$ is a smooth function, $\frac{h(r,w)}{r}$ is a locally Lipschitz function in $\left\{(r,w)~|~r> 0,~w>\tan\frac{G(r)-\frac{\pi}{2}}{n-1}\right\}$.
Since $W=\tan\frac{G(0)}{n}$ is a constant solution to the equation in $[0,1]$, we choose initial value as in
\eqref{equ-origin-quadratic} and only need to
prove that the solution exists on $[0,+\infty)$. By the monotonicity of $h(r,w)$ and the second equality in   \eqref{equ-Property-h-1}, we have
  \begin{equation}\label{equ-temp-23}
  W'(r)=\frac{h(r,W(r))}{r}
  \left\{
  \begin{array}{llll}
  <0,&\text{if }W(r)>\tan\frac{G(r)}{n},\\
  >0, & \text{if }W(r)<\tan\frac{G(r)}{n}.
  \end{array}
  \right.
  \end{equation}
By \eqref{equ-Property-h-1} and the monotonicity of $G$,  we have that $W^-:=\tan\frac{G(r)}{n}$ satisfies
\[
(W^-)'\leq 0=\dfrac{h(r,W^-)}{r}\quad\text{in }r>0.
\]
By the smoothness and monotonicity of $h(r,w)$,
for any $T\geq 1$ such that $W$ exists on $[0,T]$,
 there exists $C>0$ such that   $V:=W-W^-$ satisfies
\[
V'\geq \dfrac{h(r,W)}{r}-\dfrac{h(r,W^{-})}{r}=\dfrac{1}{r}\int_{W^-}^W\dfrac{\partial h}{\partial w}(r,w)\mathrm dw\geq  -C\dfrac{V}{r}
\]
in $(0,T)$.
Thus $r^CV(r)$ is a monotone non-decreasing function.
By the initial value $V(0)=0$, we have
\[
V\geq 0\quad\text{i.e.,}\quad W\geq W^-=\tan\frac{G(r)}{n}\quad\text{in }[0,T].
\]
Furthermore, by \eqref{equ-temp-23}, $W(r)$ is monotone non-increasing and hence
  \begin{equation}\label{equ-temp-25}
  \tan\frac{G(r)}{n}\leq W(r)\leq W(0)=\tan\frac{G(0)}{n}
  \end{equation}
in $[0,T]$.
By the Carath\'eodory extension theorem, $W$ may extend beyond $T$. Consequently $W$ exists and satisfies \eqref{equ-temp-25} on entire $[0,+\infty)$,  belongs to $C^1([0,\infty))$.

Secondly, we claim that if $W\in C^1([0,+\infty))$ is a solution to \eqref{equ-temp-11}, then $W(\frac{1}{2})=\tan\frac{G(0)}{n}$. Consequently the second equality in \eqref{equ-Property-h-1}  implies  \eqref{equ-origin-quadratic} and the uniqueness of solution follows immediately.  Arguing by contradiction, we  assume $W(\frac{1}{2})>\tan\frac{G(0)}{n}$ and the  case when $W(\frac{1}{2})<\tan\frac{G(0)}{n}$ follows similarly.
By the uniqueness result of initial value problem, we have $W(r)>\tan\frac{G(0)}{n}$ for all $r\in (0,1)$.
Consequently by \eqref{equ-temp-23}, $W(r)$ is monotone decreasing in $(0,1)$.  Hence by  the monotonicity of $h(r,w)$ and \eqref{equ-Property-h-1}, we can prove that
\[
W(r)\geq W(\frac{1}{2})\quad\text{and}\quad
h(r,W(r))\leq h(\frac{1}{2},W(\frac{1}{2}))=:-\epsilon<0\quad\text{in }(0,\frac{1}{2}].
\]
However, by the Newton--Leibnitz formula,
  \[
  W(\frac{1}{2})-W(0)=\int_0^{\frac{1}{2}}\dfrac{h(\tau,W(\tau))}{\tau}\mathrm d\tau<-\int_0^{\frac{1}{2}}\dfrac{\epsilon}{\tau}\mathrm d\tau=-\infty.
  \]
  This becomes a contradiction and finishes the proof of uniqueness.

  Thirdly, we prove that $W$ converges to $\tan\frac{G(\infty)}{n}$ at infinity.
  By   \eqref{equ-temp-23} and \eqref{equ-temp-25} from previous steps, $W$ is bounded and monotone non-increasing in $[0,+\infty)$. Consequently, as in the proof of Lemma \ref{lem-ODEsolution-subsol}, $W$ converges to a finite limit $W(\infty)=\tan\frac{G(\infty)}{n}$ at infinity and $W-\tan\frac{G(\infty)}{n}$ remains positive for sufficiently large $r$.

  Eventually, we reveal the convergence speed of $W$ at infinity.
  Let
  \[
  t:=\ln r\in (-\infty,+\infty)\quad\text{and}\quad\varphi(t):=W(r(t))-\tan\frac{G(\infty)}{n}.
  \]
  By a direct computation,
  \[
  \varphi'(t)=W'(r(t))\cdot e^t=h\left(r(t),\varphi(t)+\tan\frac{G(\infty)}{n}\right)
  =:h_1(t,\varphi)+h_2(\varphi),
  \]
  where
\[
h_1(t,\varphi):=
h\left(e^t,\varphi +\tan\frac{G(\infty)}{n}\right)-
h\left(\infty,\varphi +\tan\frac{G(\infty)}{n}\right),
\]
and
\[
h_2(\varphi):=
h\left(\infty,\varphi+\tan\frac{G(\infty)}{n}\right).
\]
By \eqref{equ-property-h-temp} and the result from previous step that $W$ converges to $\tan\frac{G(\infty)}{n}$ at infinity, there exist $R$ even larger than the one in Lemma \ref{lem-Def-h}  and $C>0$ such that
\begin{equation}\label{equ-temp-9}
e^{-\beta t}\leq h_1(t,\varphi(t))\leq C e^{-\beta t},\quad\forall~t\geq \ln R.
\end{equation}
Furthermore, by a direct computation and \eqref{equ-Property-h-3},
\begin{equation}\label{equ-temp-10}
h_2'(0)=\dfrac{\partial h}{\partial w}\left(\infty,\tan\frac{G(\infty)}{n}\right)=-n\quad\text{and hence}\quad
|h_2(\varphi)+n\varphi|=O(\varphi^2)
\end{equation}
as $\varphi\rightarrow 0.$
Applying estimates \eqref{equ-temp-9}, \eqref{equ-temp-10} and  the asymptotic stability of ODE as Lemma \ref{lem-temp-3} below,
the desired asymptotic behavior of $W$ at infinity follows from \eqref{equ-temp-21} in Lemma \ref{lem-temp-3} and this finishes the proof of this lemma.
\end{proof}

\begin{lemma}\label{lem-temp-3}
  Let $n\geq 3$, $0<\beta\leq 2$ and $\varphi$ be a non-negative solution to
  \[
  \varphi'=-n\varphi+H_1(t,\varphi)+H_2(\varphi)\quad\text{in }t>1,
  \]
  where there are constants $c,c'>0$ such that
  \[
  ce^{-\beta t}\leq H_1(t,\varphi)\leq c'e^{-\beta t},\quad\forall~t>1,
  \]
   \[
  H_2(\varphi)=O(\varphi^2)\quad\text{as }\varphi\rightarrow 0,\quad\text{and}\quad
  \varphi(t)\rightarrow 0\quad\text{as }t\rightarrow\infty.
  \]
  Then  there exist $C_1, C_2, T>0$ such that
  \begin{equation}\label{equ-temp-21}
  C_1e^{-\beta t}\leq \varphi(t)\leq C_2e^{-\beta t},\quad\forall~t>T.
  \end{equation}
\end{lemma}
\begin{proof}
By the asymptotic behavior of $\varphi$ at infinity,  there exists $T_1>1$ such that
\[
H_2(\varphi(t)) \leq\frac{1}{2}\varphi(t),\quad\forall~t>T_1.
\]
Hence
  \[
  \varphi'\leq -(n-\frac{1}{2})\varphi+c'e^{-\beta t},\quad\forall~t>T_1.
  \]
  Multiplying both sides by $e^{(n-\frac{1}{2})t}$ and taking integral over $(T_1,t)$, there exists $C>0$ such that
  \[
  0\leq \varphi\leq
  Ce^{-(n-\frac{1}{2})t}+Ce^{-\beta t}\leq Ce^{-\beta t}.
  \]
  Consequently by $H_2(\varphi)=O(\varphi^2)$ as $\varphi\rightarrow 0$,
  there exist $T_2>T_1$ and $C>0$ such that
  \[
  \left\{
  \begin{array}{llll}
  \varphi'\leq -n\varphi+c'e^{-\beta t}+Ce^{-2\beta t},\\
  \varphi'\geq -n\varphi+ce^{-\beta t}-Ce^{-2\beta t},\\
  \end{array}
  \right.\quad\forall~t>T_2.
  \]
  Multiplying both sides by $e^{nt}$ and taking integral over $(T_2,t)$, the desired estimate follows immediately.
\end{proof}

With the help of Lemmas \ref{lem-Def-h}, \ref{lem-ODEsolution-radial} and
\ref{lem-temp-3},  we have the following result on the existence of radially symmetric solution with its asymptotic behavior at infinity.
\begin{lemma}\label{lem-radialSol}
  Let $n\geq 3, 0<\beta\leq 2$ and $g(x)=G(r)$ be the radially symmetric function as in \eqref{equ-example-g}. Then any radially symmetric classical solution $u$ to \eqref{equ-Lagrangian} is of form $u=u_0+c$, $c\in \mathbb R$, where $u_0(0)=0$. Furthermore, there exist  $C_3,C_4>0$ such that
  \begin{equation}\label{equ-Asymptotic-u0}
  C_3k(x)\leq u_0(x)- \frac{1}{2}\tan\left(\frac{G(\infty)}{n}\right)|x|^2
  \leq C_4k(x),
  \end{equation}
  for sufficiently large $|x|$, where
  \[
  k(x)=\left\{
  \begin{array}{llll}
    |x|^{2-\beta}, & \text{if }\beta\neq 2,\\
    \ln|x| , & \text{if }\beta=2.\\
  \end{array}
  \right.
  \]
\end{lemma}
\begin{proof}
  Let $W$ be the solution to \eqref{equ-temp-11} in Lemma \ref{lem-ODEsolution-radial}. From the choice of $h(r,w)$ from \eqref{equ-temp-5} in Lemma \ref{lem-Def-h}, $W$ satisfies equation \eqref{equ-temp-20}.
Choose
\[
u_0(x)=\int_0^{|x|}\tau\cdot W(\tau)\mathrm d\tau,\quad\text{then}\quad
\lambda(D^2u_0)=\left(W+rW',W,\cdots,W\right),\quad\forall~|x|>0.
\]
Thus $u_0$ is the unique radially symmetric solution to \eqref{equ-Lagrangian} in $\mathbb R^n\setminus\{0\}$ with $u_0(0)=0$, and all radially symmetric solutions are characterised by $u_0+c, c\in\mathbb R.$

Especially by \eqref{equ-origin-quadratic}, $W(r)$ remains a constant for $0\leq r\leq 1$ and hence $u_0$ is a quadratic function in $B_1$. Consequently, $u_0$ is a solution to \eqref{equ-Lagrangian} in $\mathbb R^n.$
By the results in Lemma \ref{lem-ODEsolution-radial}, the asymptotic behavior \eqref{equ-Asymptotic-u0} of $u_0$ follows immediately.
\end{proof}

\begin{proof}[Proof of Theorem \ref{thm-Optimality}]
Suppose there exists a classical solution $u$ to \eqref{equ-Lagrangian} in $\mathbb R^n$ satisfying asymptotic behavior \eqref{equ-ContraAssum} with $A=\tan \frac{G(\infty)}{n} I$ i.e.,
\[
u(x)- \frac{1}{2}\tan\left(\frac{G(\infty)}{n}\right)|x|^2 =o(1),\quad\text{as }|x|\rightarrow\infty.
\]

To start with, we prove that after a rotation, the function remains a solution to \eqref{equ-Lagrangian} in $\mathbb R^n$ with the same asymptotic behavior as $u$.
For any orthogonal matrix $Q$, we take $u_Q(x):=u(Qx)$, which satisfies
\[
D^2u_Q(x)=Q^TD^2u(Qx)Q\quad\text{and}\quad\lambda(D^2u_Q(x))=\lambda(D^2u(Qx))\quad
\text{in }x\in\mathbb R^n.
\]
Thus $u_Q$ satisfies
\[
F(D^2u_Q(x))=\sum_{i=1}^n\arctan\lambda_i(D^2u_Q(x))=g(Qx)=g(x)\quad\text{in }x\in \mathbb R^n,
\]
with asymptotic behavior
\[
u_Q(x)- \dfrac{1}{2}\tan\left(\frac{G(\infty)}{n}\right)|x|^2 =o(1),\quad\text{as }|x|\rightarrow\infty.
\]

Furthermore, we prove that $u$ is radially symmetric and hence the asymptotic behavior contradicts to \eqref{equ-Asymptotic-u0} in Lemma \ref{lem-radialSol}.
  For any $\epsilon>0$, by the asymptotic behavior of $u$ and $u_Q$, there exists $R>0$ such that
  \[
  \left|u(x)-u_Q(x)\right|<\epsilon,\quad\forall~|x|=R.
  \]
  By maximum principle such as Theorem 17.1 in \cite{Book-Gilbarg-Trudinger}, we have
  \[
  \left|u(x)-u_Q(x)\right|<\epsilon,\quad\forall~|x|\leq R.
  \]
  By the arbitrariness of $\epsilon>0$, we have $u=u_Q$ for all orthogonal matrix $Q$. Consequently $u$ is radially symmetry and this finishes the proof since the asymptotic behavior of $u$ contradicts to the result in Lemma \ref{lem-radialSol}.
\end{proof}

\small

\bibliographystyle{plain}

\bibliography{C:/Bib/Thesis}

\begin{thebibliography}{10}

\bibitem{Bao-Li-Li-2014}
Jiguang Bao, Haigang Li, and Yanyan Li.
\newblock On the exterior {D}irichlet problem for {H}essian equations.
\newblock {\em Transactions of the American Mathematical Society},
  366(12):6183--6200, 2014.

\bibitem{Bao-Li-Zhang-ExteriorBerns-MA}
Jiguang Bao, Haigang Li, and Lei Zhang.
\newblock Monge-{A}mp\`ere equation on exterior domains.
\newblock {\em Calculus of Variations and Partial Differential Equations},
  52(1-2):39--63, 2015.

\bibitem{Bao-Li-Zhang-GlobalExterioDP-Dim2}
Jiguang Bao, Haigang Li, and Lei Zhang.
\newblock Global solutions and exterior {D}irichlet problem for
  {M}onge-{A}mp\`ere equation in {$\Bbb{R}^2$}.
\newblock {\em Differential and Integral Equations. An International Journal
  for Theory \& Applications}, 29(5-6):563--582, 2016.

\bibitem{Bao-Xiong-Zhou-ExistenceEntireMA}
Jiguang Bao, Jingang Xiong, and Ziwei Zhou.
\newblock Existence of entire solutions of {M}onge-{A}mp\`ere equations with
  prescribed asymptotic behavior.
\newblock {\em Calculus of Variations and Partial Differential Equations},
  58(6):Paper No. 193, 12, 2019.

\bibitem{Bhattacharya-Dirichlet-LagMeanCurva}
Arunima Bhattacharya.
\newblock The {D}irichlet problem for {L}agrangian mean curvature equation.
\newblock {\em arXiv. 2005.14420}, 2020.

\bibitem{Bhattacharya-Monney-Shankar-GradientEsti-Critical}
Arunima Bhattacharya, Connor Monney, and Ravi Shankar.
\newblock Gradient estimates for the {L}agrangian mean curvature equation with
  critical and supercritical phase.
\newblock {\em arXiv. 2205.13096}, 2022.

\bibitem{Book-Bodine-Lutz-AsymptoticIntegration}
Sigrun Bodine and Donald~A. Lutz.
\newblock {\em Asymptotic integration of differential and difference
  equations}, volume 2129 of {\em Lecture Notes in Mathematics}.
\newblock Springer, Cham, 2015.

\bibitem{Caffarelli-Li-ExtensionJCP}
Luis Caffarelli and Yanyan Li.
\newblock An extension to a theorem of {J}\"{o}rgens, {C}alabi, and
  {P}ogorelov.
\newblock {\em Communications on Pure and Applied Mathematics}, 56(5):549--583,
  2003.

\bibitem{Caffarelli-Li-Liouville-MA-periodic}
Luis Caffarelli and Yanyan Li.
\newblock A {L}iouville theorem for solutions of the {M}onge-{A}mp\`ere
  equation with periodic data.
\newblock {\em Annales de l'Institut Henri Poincar\'{e}. Analyse Non
  Lin\'{e}aire}, 21(1):97--120, 2004.

\bibitem{Caffarelli-Li-Nirenberg--ViscositySol-III}
Luis Caffarelli, Yanyan Li, and Louis Nirenberg.
\newblock Some remarks on singular solutions of nonlinear elliptic equations
  {III}: viscosity solutions including parabolic operators.
\newblock {\em Communications on Pure and Applied Mathematics}, 66(1):109--143,
  2013.

\bibitem{Caffarelli-Nirenberg-Spruck-DirichletIII}
Luis Caffarelli, Louis Nirenberg, and Joel Spruck.
\newblock The {D}irichlet problem for nonlinear second-order elliptic
  equations. {III}. {F}unctions of the eigenvalues of the {H}essian.
\newblock {\em Acta Mathematica}, 155(3-4):261--301, 1985.

\bibitem{Caffarelli-Tang-Wang-GlobalRegular-MA}
Luis~A. Caffarelli, Lan Tang, and Xu-Jia Wang.
\newblock Global {$C^{1,\alpha}$} regularity for {M}onge-{A}mp\`ere equation
  and convex envelope.
\newblock {\em Archive for Rational Mechanics and Analysis}, 244(1):127--155,
  2022.

\bibitem{Book-Coddington-Levinson-ODE}
Earl~A. Coddington and Norman Levinson.
\newblock {\em Theory of ordinary differential equations}.
\newblock McGraw-Hill Book Company, Inc., New York-Toronto-London, 1955.

\bibitem{Collins-Picard-Wu-ConcavityLagrangian}
Tristan~C. Collins, Sebastien Picard, and Xuan Wu.
\newblock Concavity of the {L}agrangian phase operator and applications.
\newblock {\em Calculus of Variations and Partial Differential Equations},
  56(4):Paper No. 89, 22, 2017.

\bibitem{User'sGuide-ViscositySol}
Michael~G. Crandall, Hitoshi Ishii, and Pierre-Louis Lions.
\newblock User's guide to viscosity solutions of second order partial
  differential equations.
\newblock {\em American Mathematical Society. Bulletin. New Series},
  27(1):1--67, 1992.

\bibitem{Book-Gilbarg-Trudinger}
David Gilbarg and Neil~S. Trudinger.
\newblock {\em Elliptic partial differential equations of second order}.
\newblock Classics in Mathematics. Springer-Verlag, Berlin, 2001.
\newblock Reprint of the 1998 edition.

\bibitem{Harvey-Lawson-CalibratedGeometries}
Reese Harvey and H.~Blaine Lawson, Jr.
\newblock Calibrated geometries.
\newblock {\em Acta Mathematica}, 148:47--157, 1982.

\bibitem{Book-Horn-Johnson-MatrixAnalysis}
Roger~Alan Horn and Charles~Royal Johnson.
\newblock {\em Matrix analysis}.
\newblock Cambridge University Press, Cambridge, second edition, 2013.

\bibitem{Jia-Xiaobiao-AsymGeneralFully}
Xiaobiao Jia.
\newblock Asymptotic behavior of solutions of fully nonlinear equations over
  exterior domains.
\newblock {\em Comptes Rendus Math\'{e}matique. Acad\'{e}mie des Sciences.
  Paris}, 358(11-12):1187--1197, 2020.

\bibitem{Jia-Li-AsympMA-halfspace}
Xiaobiao Jia and Dongsheng Li.
\newblock The asymptotic behavior of viscosity solutions of {M}onge-{A}mp\`ere
  equations in half space.
\newblock {\em Nonlinear Analysis. Theory, Methods \& Applications. An
  International Multidisciplinary Journal}, 206:Paper No. 112229, 23, 2021.

\bibitem{Jia-Li-Li-AsympMA-halfspace}
Xiaobiao Jia, Dongsheng Li, and Zhisu Li.
\newblock Asymptotic behavior at infinity of solutions of {M}onge-{A}mp\`ere
  equations in half spaces.
\newblock {\em Journal of Differential Equations}, 269(1):326--348, 2020.

\bibitem{Li-Li-Yuan-BernsteinThm}
Dongsheng Li, Zhisu Li, and Yu~Yuan.
\newblock A {B}ernstein problem for special {L}agrangian equations in exterior
  domains.
\newblock {\em Advances in Mathematics}, 361:106927, 29, 2020.

\bibitem{Li-Wang-EntireDirichlet}
Xiaoliang Li and Cong Wang.
\newblock On the exterior {D}irichlet problem for {H}essian-type fully
  nonlinear elliptic equations.
\newblock {\em Communications in Contemporary Mathematics}, Paper No. 2250082,
  2023.

\bibitem{Li-Lu-ExistandNonExist}
Yanyan Li and Siyuan Lu.
\newblock Existence and nonexistence to exterior {D}irichlet problem for
  {M}onge-{A}mp\`ere equation.
\newblock {\em Calculus of Variations and Partial Differential Equations},
  57(6):Paper No. 161, 17, 2018.

\bibitem{Li-Dirichlet-SPL}
Zhisu Li.
\newblock On the exterior {D}irichlet problem for special {L}agrangian
  equations.
\newblock {\em Transactions of the American Mathematical Society},
  372(2):889--924, 2019.

\bibitem{Liu-Bao-2021-Expansion-LagMeanC}
Zixiao Liu and Jiguang Bao.
\newblock Asymptotic expansion at infinity of solutions of {M}onge-{A}mp\`ere
  type equations.
\newblock {\em Nonlinear Analysis. Theory, Methods \& Applications. An
  International Multidisciplinary Journal}, 212:Paper No. 112450, 17, 2021.

\bibitem{Liu-Bao-2021-Dim2}
Zixiao Liu and Jiguang Bao.
\newblock Asymptotic expansion and optimal symmetry of minimal gradient graph
  equations in dimension 2.
\newblock {\em Communications in Contemporary Mathematics}, Paper No. 2150110,
  25, 2022.

\bibitem{Liu-Bao-2020-ExpansionSPL}
Zixiao Liu and Jiguang Bao.
\newblock Asymptotic expansion at infinity of solutions of special {L}agrangian
  equations.
\newblock {\em Journal of Geometric Analysis}, 32(3):Paper No. 90, 34, 2022.

\bibitem{Liu-Bao-2021-Dim2-MeanCur}
Zixiao Liu and Jiguang Bao.
\newblock Asymptotic expansion of 2-dimensional gradient graph with vanishing
  mean curvature at infinity.
\newblock {\em Communications on Pure and Applied Analysis}, 21(9):2911--2931,
  2022.

\bibitem{Lu-Dirichlet-Lagrangian}
Siyuan Lu.
\newblock On the {D}irichlet problem for {L}agrangian phase equation with
  critical and supercritical phase.
\newblock {\em arXiv. 2204.05420}, 2022.

\bibitem{Savin-MA-PointWiseBoundaryEst}
Ovidiu~V. Savin.
\newblock Pointwise {$C^{2,\alpha}$} estimates at the boundary for the
  {M}onge-{A}mp\`ere equation.
\newblock {\em Journal of the American Mathematical Society}, 26(1):63--99,
  2013.

\bibitem{Savin-LocalizationThm}
Ovidiu~V. Savin.
\newblock A localization theorem and boundary regularity for a class of
  degenerate {M}onge-{A}mp\`ere equations.
\newblock {\em Journal of Differential Equations}, 256(2):327--388, 2014.

\bibitem{Trudinger-Wang-BoundaryRegularity}
Neil~S. Trudinger and Xu-Jia Wang.
\newblock Boundary regularity for the {M}onge-{A}mp\`ere and affine maximal
  surface equations.
\newblock {\em Annals of Mathematics. Second Series}, 167(3):993--1028, 2008.

\bibitem{Chong-Rongli-Bao-SecondBoundary-SPL}
Chong Wang, Rongli Huang, and Jiguang Bao.
\newblock On the second boundary value problem for {L}agrangian mean curvature
  equation.
\newblock {\em Calculus of Variations and Partial Differential Equations},
  62(3):Paper No. 74, 2023.

\bibitem{Yuan-GlobalSolution-SPL}
Yu~Yuan.
\newblock Global solutions to special {L}agrangian equations.
\newblock {\em Proceedings of the American Mathematical Society},
  134(5):1355--1358, 2006.

\end{thebibliography}

\bigskip

\noindent J. Bao*, C. Wang

\medskip

\noindent  School of Mathematical Sciences, Beijing Normal University\\
Laboratory of Mathematics and Complex Systems, Ministry of Education\\
Beijing 100875, China \\[1mm]
Email: \textsf{jgbao@bnu.edu.cn, cwang@mail.bnu.edu.cn}

\medskip

\noindent Z. Liu

\medskip

\noindent Institute of Applied Mathematics, Department of Mathematics, Faculty of Science\\
Beijing University of Technology\\
Beijing 100124, China \\[1mm]
Email: \textsf{liuzixiao@bjut.edu.cn}

\end{document}